\documentclass[11pt]{amsart}
\usepackage[utf8]{inputenc}
\usepackage[T1]{fontenc}
\usepackage[a4paper, margin = 2.8cm]{geometry}
\usepackage{baskervillef}
\usepackage[noadjust]{cite}
\usepackage{amsmath}
\usepackage{nccmath}
\usepackage{mathrsfs}
\usepackage[english]{babel}
\usepackage{amssymb}
\usepackage{amsthm}
\usepackage{bm}
\usepackage{tikz}
\usetikzlibrary{positioning,decorations.markings,backgrounds,fit,shapes.geometric,graphs}
\usepackage{tikz-cd}
\usepackage{tikzsymbols}
\usepackage{mdframed}
\usepackage{framed}
\usepackage{mathtools}
\usepackage{fancyhdr}
\usepackage{enumerate}
\usepackage{graphicx}
\usepackage{hyperref}
\usepackage{cleveref}
\usepackage{colortbl}
\usepackage{epigraph}
\usepackage{subfig}
\usepackage{gensymb}
\usepackage{comment}
\usepackage[font=small, labelfont=bf]{caption}
\usepackage{chngcntr}
\usepackage{xcolor}
\usepackage{graphicx}
\usepackage{framed}
\usepackage[normalem]{ulem}
\usepackage{mathdots}
\usepackage[symbol]{footmisc}
\usepackage{booktabs} 

\theoremstyle{plain}
\newtheorem{proposition}{Proposition}[section]
\newtheorem{corollary}[proposition]{Corollary}
\newtheorem{lemma}[proposition]{Lemma}

\newtheorem{theorem}[proposition]{Theorem}

\theoremstyle{definition}
\newtheorem{definition}[proposition]{Definition}
\newtheorem{example}[proposition]{Example}
\newtheorem*{example*}{Example}

\theoremstyle{remark}
\newtheorem*{notation}{Notation}
\newtheorem{remark}[proposition]{Remark}
\newtheorem*{remark*}{Remark}

\numberwithin{equation}{section}
\numberwithin{figure}{section}

\newcommand*\openquote{\makebox(25,-22){\scalebox{5}{\color{white}``}}}
\newcommand*\closequote{\makebox(25,-22){\scalebox{5}{\color{white}''}}}
\colorlet{shadecolor}{gray!20}

\makeatletter
\newif\if@right
\def\shadequote{\@righttrue\shadequote@i}
\def\shadequote@i{\begin{snugshade}\begin{quote}\openquote}
\def\endshadequote{%
  \if@right\hfill\fi\closequote\end{quote}\end{snugshade}}
\@namedef{shadequote*}{\@rightfalse\shadequote@i}
\@namedef{endshadequote*}{\endshadequote}
\makeatother

\DeclareMathOperator{\Ker}{Ker}

\DeclareMathOperator{\onto}{\twoheadrightarrow}
\DeclareMathOperator{\into}{\hookrightarrow}

\DeclarePairedDelimiter\abs{\lvert}{\rvert}

\DeclareMathOperator{\Aut}{Aut}

\DeclareMathOperator{\GL}{GL}
\DeclareMathOperator{\SL}{SL}

\DeclareMathOperator{\htt}{ht}

\DeclareMathOperator{\mfG}{\mathfrak{G}}
\DeclareMathOperator{\mfU}{\mathfrak{U}}

\DeclareMathOperator{\aff}{aff}

\newcommand\mc{\mathcal}
\newcommand\ov{\overline}
\newcommand\wt{\widetilde}

\DeclareMathOperator{\R}{\mathbb{R}}
\DeclareMathOperator{\K}{\mathbb{K}}
\DeclareMathOperator{\Z}{\mathbb{Z}}

\DeclareMathOperator{\N}{\mathbb{N}}

\DeclareMathOperator{\F}{\mathbb{F}}

\title{Presentation of Borel subgroups of Kac--Moody groups over local rings}
\author{Robynn Corveleyn}
\thanks{The author is grateful for financial support from the FWO and the F.R.S.–FNRS under the Excellence of Science (EOS) program (project ID 40007542).}
\address{Robynn Corveleyn ({\tt robynn.corveleyn@uclouvain.be}) \newline Institut de Recherche en Math\'{e}matique et Physique, UCLouvain, Chemin du Cyclotron 2, 1348 Louvain-la-Neuve, Belgium.}
\begin{document}
\begin{abstract}
We obtain presentations of Borel subgroups of certain Chevalley groups and Kac--Moody groups of simply laced type. More precisely, we show that Borel subgroups of simply laced Chevalley groups over a commutative, unital ring (satisfying a minor technical condition), are the amalgamated product of their rank $2$ subgroups, generalising a result of Tits over fields (different from $\mathbb{F}_2$). We moreover show that Borel subgroups of Kac--Moody groups over local Bézout domains also have a presentation as an amalgamated product of their rank $2$ subgroups as soon as this holds for their rank $3$ subgroups. In particular we deduce that Borel subgroups of simply laced Kac--Moody groups over a local Bézout domain (with residue field different from $\mathbb{F}_2$) have such a presentation if and only if they are of $3$-spherical type. 
\end{abstract}
\maketitle
\thispagestyle{empty}

\vspace{-12pt}
\section{Introduction}
\begingroup
\renewcommand\theproposition{\Alph{proposition}}
Kac--Moody groups are natural infinite-dimensional generalisations of semisimple algebraic groups, and have become a prominent class of objects of study in several domains, such as geometric group theory, representation theory and algebraic geometry (see e.g.\ \cite{MarquisKMGbook}).

A (minimal) Kac--Moody group $\mfG_A(R)$ over a ring $R$ is constructed from a generalised Cartan matrix (GCM) $A$. More precisely, associated to any GCM $A$, there is a Coxeter group $\mc{W}$ (the Weyl group of $A$), and the group $\mfG_A(R)$ is defined by a Steinberg-like presentation in terms of generators $x_{\alpha}(r)$, for all $r \in R$ and all roots $\alpha$ in the root system $\Phi$ of the Coxeter complex of $\mc{W}$. If $\mc{W}$ is infinite (or, equivalently, $A$ is not a Cartan matrix), there are infinitely many roots $\alpha$, and hence $\mfG_A(R)$ is defined by a presentation in terms of infinitely many one-parameter subgroups $x_{\alpha}(R)$ of generators (and infinitely many relations).

Inside $\mfG_A(R)$, there is a Borel subgroup $B^+_R$, consisting of a “maximal unipotent” subgroup $U^+_R$ generated by the elements $x_{\alpha}(r)$ for all $\alpha \in \Phi$ positive, as well as a torus $T_R$. 

In the literature, the question of obtaining tractable presentations of the groups $U^+_R$ (and hence $B^+_R$) has mostly been considered over fields. In \cite[Proposition 7]{Tits86} (see also \Cref{reduction step Tits} below), Tits proved that over most rings $R$, the maximal unipotent subgroup $U^+_{R}$ of a Chevalley group (i.e.\ a Kac--Moody group of \emph{spherical} type, meaning that $A$ is a Cartan matrix) is the amalgamated product of its rank $2$ subgroups as soon as this holds for all of its rank $3$ subgroups. He moreover showed that when $\K$ is a field different from $\F_2$, the group $U^+_{\K}$ of type $A_3$ (i.e.\ the group of upper triangular unipotent matrices in $\SL_4(\K)$) has such a presentation.

Later, using geometric arguments, Abramenko obtained in \cite[Proposition 11]{Abramenko} the same result for the maximal unipotent subgroup $U^+_{\K}$ of Chevalley groups of type $B_3$ and $C_3$ over fields $\K$, provided that $\K$ has size at least $16$. This bound was improved to $5$ in \cite{HornNessVMDG}. In combination with the reduction step in Tits' work, this effectively settled the question of obtaining such presentations of the subgroup $U^+_{\K}$ of Chevalley groups over fields $\K$.

Over more general rings, the same question remains open. In this paper, we adapt Tits' strategy in \cite{Tits86} to prove the following. For any root $\alpha$, denote by $U_{\alpha}$ the subgroup $U_{\alpha} \cong \{x_{\alpha}(r) \mid r \in R\}$ of $U^+_R$. Recall that a GCM $A = (A_{ij})_{i,j\in I}$ is called \emph{simply laced} if $A_{ij}A_{ji} \leq 1$ for all distinct $i,j \in I$.

\begin{theorem}
\label{spherical main}
Let $A = (A_{ij})_{i,j \in I}$ be a simply laced Cartan matrix. Let $R$ be a commutative, unital ring such that $1_R$ is the sum of $2$ invertible elements. Then $U^+_R \leqslant \mfG_A(R)$ is the amalgamated product of its subgroups $U_{i,j} :=  \langle U_{\alpha_i}, U_{\alpha_j} \rangle$, along the canonical inclusions $U_{\alpha_i} \into U_{i,j}$, for all $i,j \in I$.
\end{theorem}
Note that a necessary condition for the conclusion of \Cref{spherical main} to hold is that $R$ has no quotient $\F_2$; this latter condition is in general weaker than $1_R$ being the sum of two invertible elements, but is for instance equivalent to it when $R$ is a semi-local ring (see \Cref{exists element a+1 a2+1} and \Cref{rk F2} below).

%

In \cite{DevillersMu}, similar results for Kac--Moody groups of non-spherical type over fields $\K$ are obtained. Using geometric methods rather than group theoretic ones as in Tits' work, the authors again make a reduction step, showing that $U^+_{\K}$ has a presentation in terms of an amalgamated product of its rank $2$ subgroups, if this is the case for each rank $3$ subgroup. Using then the aforementioned results for the $A_3$ (\cite{Tits86}), $B_3$ and $C_3$ cases (\cite{Abramenko} and \cite{HornNessVMDG}), this yields the desired presentation for $U^+_{\K} \leqslant \mfG_A(\K)$, with $A$ of $3$-spherical type, and $\K$ a field of size at least $5$. We recall that a GCM $A = (A_{ij})_{i,j \in I}$ is \emph{$3$-spherical} if $A_J := (A_{ij})_{i,j \in J}$ is a Cartan matrix for all $J\subseteq I$ with $\abs{J} = 3$.

It was moreover recently shown in \cite[Theorem E]{CapMar26} that $A$ being of $3$-spherical type is a necessary condition for $U^+_{\K}$ to have this presentation. These results together provide a reasonably good understanding of presentations of $U^+_{\K}$ in the Kac--Moody setting over fields. Over more general rings, however, this is no longer the case. 

In this paper, we again carry out a reduction step to the rank $3$ case, as in \cite{DevillersMu}, this time for local Bézout domains $R$. We will establish our results for the group \[\ov{U}^+_R = \varphi_R(U^+_R) \leqslant~\mfG_A(\K),\] where $\K$ is the field of fractions of $R$ and $\varphi_R$ is the group homomorphism
\[\varphi_R \colon \mfG_A(R) \to \mfG_A(\K),\]
naturally induced by the inclusion $R \into \K$. It was recently proven in \cite{MMinject} that, in this context, $\varphi_R$ is actually injective, and hence $\ov{U}^+_R \cong U^+_R$. However, we will not need this fact, and in fact our methods provide an alternative proof of the main result of \emph{loc.\ cit.\ }in the $3$-spherical simply laced case (see \Cref{corollary injective} below).

For each $\alpha \in \Phi$, we again write $U_{\alpha}$ for its image in $\mfG_A(\K)$, and we write $\ov{U}_{i,j}$ for $\varphi_R(\langle U_{\alpha_i} , U_{\alpha_j}\rangle)$.
The reduction step to the rank $3$ case which we obtain is then as follows. 

\begin{theorem}
\label{main theorem}
Let $(A_{ij})_{i,j\in I}$ be a $3$-spherical generalised Cartan matrix. Let $R$ be a local Bézout domain satisfying \eqref{co}. Suppose that for every $J \subseteq I$ with $\abs{J} = 3$, the subgroup
\[U^+_{J R} := \mfG_{A_{J}}(R) \cap U^+_R  \] 
of $U^+_R$ is the amalgamated product of its subgroups $U_{j_1, j_2} := \langle U_{\alpha_{j_1}}, U_{\alpha_{j_2}}\rangle$ along the canonical inclusions $U_{\alpha_{j_1}} \into U_{j_1, j_2}$, for all $j_1, j_2 \in J$. Then 
$\ov{U}^{+}_R$ is the amalgamated product of its subgroups $\ov{U}_{i_1,i_2}$ along the inclusions $U_{\alpha_{i_1}} \into \ov{U}_{i_1, i_2}$, for all  $i_1,i_2 \in I$.
\end{theorem}
The condition \eqref{co} is the following mild assumption on $R$.
\begin{align}
\tag{co} R \: &  \text{ has }  \text{no quotient } \F_2  \text{ if } A_{ij}A_{ji} = 2 \text{ for some } i,j \in  I, \text{ and } R \text{ has no quotient } \F_2 \text{ or } \F_3 \label{co} \\  \text{ if } & A_{ij}A_{ji} = 3 \text{ for some } i,j \in I.  \notag
\end{align}

In combination with \Cref{spherical main}, we then obtain a presentation for the subgroups $\ov{U}^+_R$ of $3$-spherical simply laced Kac--Moody groups.

\begin{corollary}
\label{corollary C}
Let $R$ be a local Bézout domain such that its residue field is not $\F_2
$. Let $(A_{ij})_{i,j \in I}$ be a simply laced generalised Cartan matrix. Then $\ov{U}^+_R \leqslant \mfG_A^{\min}(R)$ is the amalgamated product of its subgroups $\ov{U}_{i_1, i_2}$ along the canonical inclusions $U_{\alpha_{i_1}} \into \ov{U}_{i_1, i_2}$, for all $i_1, i_2 \in I$, if and only if $A$ is $3$-spherical. 
\end{corollary}

As alluded to above, \Cref{corollary C} implies that under the same hypotheses, the canonical map $\varphi_R \colon \mfG_A(R) \to \mfG_A(\K)$ is injective, thus providing an alternative proof in the $3$-spherical simply laced case of the recent result \cite[Theorem A]{MMinject} that $\varphi_R$ is injective.

\begin{corollary}
\label{corollary injective}
Let $A$ be a $3$-spherical simply laced GCM and let $R$ be a local Bézout domain such that its residue field is different from $\F_2$. Then the canonical map 
\[\varphi_R \colon \mfG_{A}(R) \to \mfG_{A}(\K)\]
is injective.
\end{corollary}

The first part of the paper is dedicated to the proof of \Cref{main theorem}. We now briefly outline some of the methods we will use to prove this result.

To obtain their presentation of $U^+_{\K}$ for $A$ non-spherical and $\K$ a field, the authors of \cite{DevillersMu} make use of \cite[Corollaire 1]{Tits86} (sometimes referred to as the \emph{Tits lemma}) which, for groups $G$ acting transitively on a chamber system $C$, characterises the existence of a presentation of $G$ as the amalgamated product of the stabilisers of rank $2$ residues in $C$, in terms of the simple connectedness of $C$. 

The authors of \cite{DevillersMu} moreover introduce \emph{filtrations} of chamber systems $C$, meaning a sequence $C_0 \subseteq C_1 \subseteq C_2~\subseteq~\ldots$ of subcomplexes of $C$ such that $\bigcup_{n \in \N} C_n = C$, satisfying some additional properties, and show that under certain conditions, the simple connectedness of $C$ is equivalent to the simple connectedness of $C_0$ (see Theorem 3.14 in \emph{loc.\ cit.\ }or \Cref{314 devmu} below).

Since Kac--Moody groups over fields $\K$ have an associated building $\mc{C}^-(\mfG_A(\K)) = \mfG_A(\K)/B^-_{\K}$, which is a simply connected chamber system, the strategy of \emph{loc.\ cit.\ }is to construct a suitable filtration of $\mc{C}^-(\mfG_A(\K))$ such that $U^+_{\K}$ acts transitively on $C_0$. A condition on $(C_n)_{n \in \N}$ required to apply their result and conclude that $C_0$ is simply connected, amounts to the rank $3$ residues of $C_0$ being simply connected. In turn, by the Tits lemma, this condition is equivalent to the subgroups $U^+_{J \K}$ of $U^+_{\K}$ with $\abs{J} = 3$ admitting a presentation as an amalgamated product of their rank $2$ subgroups. 

For any domain $R$ with field of fractions $\K$, the Kac--Moody group $\mfG_A^{\min}(R) := \varphi_R(\mfG_A(R)) \leqslant \mfG_A(\K)$ also has an associated chamber system $\mc{C}^-(\mfG^{\min}_A(R))$ on which it acts transitively. When $R$ is a  Bézout domain, it is moreover isomorphic to the building $\mc{C}^-(\mfG_A(\K))$, by \cite[Lemma 8.7]{MMinject}. Hence in this case, $\mc{C}^-(\mfG^{\min}_A(R))$ is simply connected.

In the first part of this paper, we obtain a Birkhoff decomposition for $\mfG_A^{\min}(R)$ for $R$ a local ring. We then use this result to construct, under the assumptions of \Cref{main theorem}, a filtration $(C_n)_{n \in \N}$ of $\mc{C}^-(\mfG^{\min}_A(R))$, such that $\ov{U}^+_R$ acts transitively on $C_0$, and which reduces the simple connectedness of $C_0$ to that of its rank $3$ residues.

This article is structured as follows. In \cref{sect prelim}, we introduce the necessary preliminaries on Kac--Moody groups which we will use throughout the rest of the paper. In \cref{sect birkhoff} we construct a Birkhoff decomposition of $\mfG^{\min}_A(R)$, for $R$ a local ring. \Cref{sect filtration} is dedicated to the proof of \Cref{main theorem}: we first give the necessary background from \cite{DevillersMu}, and then obtain the filtration required to carry out the reduction step to the rank $3$ case, showing \Cref{main theorem}. In \cref{sect A3}, we prove \Cref{spherical main}, and in \cref{section presentation U} we conclude with the proofs of \Cref{corollary C} and \Cref{corollary injective}.

\vspace{5pt}
\noindent \textbf{Acknowledgement.}
I would like to thank Timothée Marquis for helpful comments on an earlier version of this article.

\vspace{5pt}
\noindent \textbf{Declaration on the usage of AI.}
This work, from the initial research phase until the present written article, was completed without the use of artificial intelligence at any stage.
\begin{notation}
Throughout this article, for a group $G$ and elements $x,y \in G$, we denote by $[x,y]$ the commutator $[x,y] := x^{-1} y^{-1} xy$. Moreover, we define $x^y := y^{-1} x y$. 
\end{notation}

\endgroup
\setcounter{proposition}{0}
\section{Preliminaries}
\label{sect prelim}
\subsection{Coxeter groups}
\label{sect coxeter}
For a reference on the notions introduced in this section, we refer to \cite[Chapter 2]{AbramenkoBrown}. A \emph{Coxeter system} consists of a pair $(W,S)$, where $S$ is a (finite) set of involutions which generate the \emph{Coxeter group} $W$, given by a presentation
\[W := \left\langle s_i\in S \mid (s_i s_j)^{m_{ij}} = 1 \right\rangle,\]
where $m_{ij} = m_{ji} \in \N_{\geq 2} \cup \{\infty\}$ if $i \neq j$ and $m_{ii} = 1$ for all $i$. Let $n = \abs{S}$. We call $(m_{ij})_{1\leq i,j \leq n}$ the \emph{Coxeter matrix} associated to $(W,S)$. For any $w \in W$, there is a uniquely defined \emph{length} $\ell(w)$ defined as $d\in \N$ if $w~=~s_{i_1} \ldots s_{i_d}$ and $d$ is minimal for this property. We say that the \emph{word} $\mathbf{w} = (s_{i_1}, \ldots, s_{i_d})$ on the alphabet $S$ is \emph{reduced} if the element $w := s_{i_1} \ldots s_{i_d}\in W$ it represents has length $d$.

By Tits' solution to the word problem for Coxeter groups (see for example \cite[Theorem 2.33]{AbramenkoBrown}), any word $(s_{j_1} , \ldots, s_{j_m})$ representing a given element of $W$ may be rewritten in reduced form by carrying out a sequence of the following two operations: 
\begin{enumerate}
\item[(M1)] Deleting a subword of the form $(s_i,s_i)$. 
\item[(M2)] \hypertarget{M2}{Replacing a  subword $\mathbf{t}(i,j) := (s_i, s_j, s_i, \ldots)$ of length $m_{ij} < \infty$ by the word $\mathbf{t}(j,i) := (s_j, s_i, s_j,\ldots)$ of length $m_{ji} = m_{ij}$.}
\end{enumerate}
We will denote by $\tau(i,j)$ the element $s_i s_j s_i \ldots \in W$ represented by the word $\mathbf{t}(i,j)$. 

A Coxeter system $(W,S)$ has a naturally associated simplicial complex, called a \emph{Coxeter complex}, on which $W$ acts. It comes equipped with a set $\Phi$ of \emph{roots} or \emph{half spaces}, consisting of \emph{simple roots} $\alpha_{s_i}$ for $s_i \in S$, and their images $w(\alpha_{s_i})$ for all $w \in W$. The root system $\Phi$ decomposes into a union $\Phi = \Phi^+ \sqcup \Phi^-$ of subsets $\Phi^+$ and $\Phi^-$ consisting of respectively \emph{positive} and \emph{negative} roots. The sign of roots is characterised as follows.
\begin{lemma}[{\cite[Theorem 5.4]{humphreysRefl}}]
\label{length ws implies root positive}
Let $(W,S)$ be a Coxeter system with root system $\Phi = \Phi^+ \sqcup~\Phi^-$. 
Then for any $w \in W$, $\ell(ws) < \ell(w)$ if and only if $w(\alpha_s) \in \Phi^-$.
\end{lemma}
Let $J \subseteq S$. We say that $w \in W$ is \emph{right $J$-reduced} if $\ell(w s_j) > \ell(w)$ for all $s_j \in J$.  We denote moreover by $W_J$ the subgroup of $W$ generated by all $s_j \in J$. For any $w \in W$, the coset $w W_J$ has a unique shortest element $w_J$. This is the unique element $w_J \in wW_J$ such that $\ell(w_J s_j) = \ell(w_J) + 1$ for all $s_j \in J$.  Therefore the element $w_J$ is right $J$-reduced. Moreover, for any $w' \in w W_J$, we have $\ell(w') = \ell(w_J) + \ell(w_J^{-1} w')$ (see \cite[Proposition 2.20]{AbramenkoBrown}).

\subsection{Generalised Cartan matrices and associated Weyl group}
Let $A = (A_{ij})_{i,j \in I}$ be a matrix with integer coefficients. It is called a \emph{generalised Cartan matrix} (GCM) if $A_{ii} = 2$, $A_{ij} \leq 0$ and if $A_{ij} = 0 \iff A_{ji} = 0$, for all $i \neq j \in I$. 

A GCM is called \emph{spherical} (equivalently, of \emph{finite type}) if it is a Cartan matrix. We say that $A$ is \emph{$n$-spherical} if $A_{J} := (A_{ij})_{i,j \in J}$ is spherical, for all $J \subseteq I$ with $\abs{J} \leq n$. The size $\abs{I}$ of $I$ is called the \emph{rank} of $A$. In this paper, all GCM's $A$ are assumed to be $2$-spherical.

A GCM is called \emph{decomposable} if it can be written as a block-diagonal matrix $\begin{psmallmatrix} A_1 & 0 \\ 0 & A_2\end{psmallmatrix}$ after reordering its index set $I$.

A GCM is called \emph{simply laced} if for all $J \subseteq I$ of rank $2$, $A_J = \begin{psmallmatrix} 2 & -1 \\ -1 & 2\end{psmallmatrix}$ or $A_J = \begin{psmallmatrix} 2 & 0 \\ 0 & 2\end{psmallmatrix}$. Simply laced GCM's are $2$-spherical. The Cartan matrix \emph{of type $A_3$} is given by 
\[A = \begin{psmallmatrix} 2 & -1 & 0 \\ -1 & 2 & -1 \\ 0 & -1 & 2\end{psmallmatrix}.\]

Let $Q := \bigoplus_{i \in I} {\Z}{\alpha_i}$ be the free abelian group generated by symbols $\alpha_i$, for all $i \in I$. They are called the \emph{simple roots} associated to $A$. We denote $\Pi := \{\alpha_i \mid i \in I \}$.
Let $\mc{W} \leqslant \GL(Q)$ be the group generated by the \emph{simple reflections} \[ s_i \colon Q \to Q, \quad  \alpha_j \mapsto \alpha_j - A_{ij}\alpha_i. \vspace{-3pt}\]  It is called the \emph{Weyl group} of $A$. 

The pair $(\mc{W},S)$, where $S := \{s_i \mid i \in I \}$, is a Coxeter system, with associated Coxeter matrix $(m_{ij})_{i,j \in I}$ where for $i \neq j$, we have $m_{ij} = 2,3,4,6$, or $\infty$, respectively if and only if $A_{ij}A_{ji} = 0,1,2,3$ or $\geq 4$. Note that the GCM $A$ is spherical if and only if $\mc{W}$ is finite, and $A$ is $2$-spherical if and only if $m_{ij} < \infty$ for all $i, j \in I$.
The set $\Phi := \mc{W} \cdot \Pi$ is called the set of \emph{real roots} associated to $A$, and it is in one-to-one correspondence with the set of roots of the Coxeter complex of $(\mc{W},S)$. Under this correspondence, given by the $\mc{W}$-equivariant extension of the map $\alpha_i \mapsto \alpha_{s_i}$ for $i \in I$, the set $\Phi^{\pm}$ is identified with $\Phi \cap \sum_{i \in I} \N \alpha_i$.

For any $J \subseteq I$, we define $\Phi^{\pm}(J) := \Phi^{\pm} \cap \bigoplus_{j \in J} \Z \alpha_j$.

A pair of distinct roots $\alpha, \beta \in \Phi$ is called \emph{prenilpotent} if there are some $w_1, w_2 \in \mc{W}$ such that $\{  w_1\alpha, w_1\beta\}\subseteq \Phi^+$ and $\{w_2   \alpha, w_2   \beta \} \subseteq \Phi^-$. Then the set $]\alpha, \beta[_{\N} :=\left(\N_{>0} \alpha +\N_{>0}\beta\right) \cap \Phi$ is finite (see \cite[Definition 7.41]{MarquisKMGbook}). The \emph{height} of a root $\alpha = \sum_{i \in I} k_i \alpha_i \in \Phi$ is the integer $\htt(\alpha) := \sum_{i \in I} k_i$.

\subsection{\texorpdfstring{The Kac--Moody groups $\mfG_{A}(R)$ and $\mfG_{A}^{\min}(R)$}{The Kac-Moody groups G(R) and Gmin(R)}}

\label{sect KM groups}
In this section, we briefly introduce the Kac--Moody groups $\mfG_{A}(R)$ and $\mfG_{A}^{\min}(R)$ under consideration in this article. We refer respectively to \cite[Chapter 7]{MarquisKMGbook} and to \cite[Section 8.8]{MarquisKMGbook} for the precise definitions of $\mfG_{A}(R)$ and $\mfG^{\min}_{A}(R)$. Throughout this section and the rest of this article, let $R$ be a commutative ring with unit. 

\vspace{-10pt}
\subsubsection*{\texorpdfstring{The groups $\mfG_{A}(R)$}{The groups G(R)}} Let $A = (A_{ij})_{i,j \in I}$ be a GCM. 
The \emph{(“simply connected”) minimal Kac--Moody group $\mfG_{A}(R)$ of type $A$ over $R$}  is given by a presentation, with generators depending on the real roots $\Phi = \mc{W} \cdot \{\alpha_i \mid i \in I\}$ associated to $A$. 
We refer to \cite[Definition 7.47 {\it \&} Example 7.11]{MarquisKMGbook} for the precise definition (see also \cite{Tits87}). For our purposes, the relevant information is the following. 

For every $\alpha \in \Phi$, we specify a copy $\mfU_{\alpha}(R) \cong (R,+)$ of the additive group of $R$, and denote the isomorphism by 
\[x_{\alpha} \colon (R,+) \to \mfU_{\alpha}(R).\]
We denote by $\mathfrak{St}_A(R)$ the group defined as the free product of all the $\mfU_{\alpha}(R)$ as $\alpha$ runs through $\Phi$, quotiented out by the relations 
\begin{equation}
\label{commutation KM min}
[x_{\alpha}(r), x_{\beta}(s)] = \prod_{\gamma}x_{\gamma}(C_{ij}^{\alpha \beta} r^{i} s^{j}), \quad \text{ for all } \{\alpha, \beta\} \subseteq \Phi \text{ prenilpotent and all } r, s \in R,
\end{equation}
where $\gamma = i \alpha + j \beta$ enumerates $]\alpha,\beta[_{\N}$ in a fixed order, and the constants $C_{ij}^{\alpha\beta}$ are integers as given by \cite[Proposition 7.43]{MarquisKMGbook}.
The group $\mfG_{A}(R)$ is then defined in terms of a presentation, generated by the group $\mathfrak{St}_A(R)$, subject to some additional relations (see \cite[Definition 7.47]{MarquisKMGbook}). 

Throughout, we set $G_R := \mfG_{A}(R)$, and distinguish the following subgroups of $G_R$.
We identify 
\[U_{\alpha} = U_{\alpha R} := \mfU_{\alpha}(R), \text{ for any } \alpha \in \Phi, \] 
with their images in $G_R$, writing again $x_{\alpha} \colon (R,+) \to U_{\alpha} \leqslant G_R$ for the induced map, and define moreover
\[
U^{\pm}_{AR} : = U^{\pm}_R := \langle U_{\alpha R} \mid \alpha \in \Phi^{\pm} \rangle \leqslant G_R.
\]
If $A = \begin{psmallmatrix} A_1 & 0 \\ 0 & A_2 \end{psmallmatrix}$ is decomposable, then $U^{\pm}_{AR}~\cong~U^{\pm}_{A_1 R}~\times~U^{\pm}_{A_2 R}$.

The following gives a normal form for elements in $U^+_{AR}$ in case $A$ is spherical. See also \cite[Theorem 5.3.3]{Carter} for a more elementary proof in the field case. Note that the existence of such an expression follows immediately by \eqref{commutation KM min}.

\begin{lemma}[{\cite[Proposition 1]{Tits87}}]
\label{unique expression U+}
Let $A$ be a Cartan matrix with associated root system $\Phi$. Then for any fixed total order $<$ on $\Phi^+$, each element $u \in U^+_{AR} $ has a unique expression of the form
\[
u = \prod_{\alpha \in \Phi^+} x_{\alpha}(r_{\alpha}), \quad \text{ with } r_{\alpha} \in R, \text{ for all } \alpha \in \Phi,
\]
where the product is taken over the roots in increasing order according to $<$. 
\end{lemma}

Inside $G_R$, there is a \emph{torus} $T_R \cong (R^{\times})^{\abs{I}}$, and it follows from the defining relations of $G_R$ that $T_R$ normalises $U^{\pm}_R$. 
We define \emph{the Borel subgroups} $B^{\pm}_R \leqslant G_R$ as 
\[B^{\pm}_R := U^{\pm}_R T_R.\]
We will write $B^{+}_2(R)$ (resp.\ $B^-_2(R)$) for the subgroup of upper (resp.\ lower) triangular matrices of $\SL_2(R)$, and $U^{\pm}_2(R)$ for the subgroup of $B^{\pm}_2(R)$ consisting of unipotent matrices.

For $i \in I$, we define moreover
\[G_{iR} := \langle U_{\alpha_i R}, U_{-\alpha_i R}\rangle \leqslant G_R \quad \text{ and } \quad P_{iR}^{\pm} := \langle G_{iR} , B^{\pm}_R\rangle \leqslant G_R,\]
and for $J\subseteq I$, 
\[ U_{JR}^{\pm} := \langle U_{\alpha R} \mid \alpha \in \Phi^{\pm}(J) \rangle, \quad  G_{JR} := \langle G_{j R} \mid j \in J\rangle, \quad \text{ and } \quad P_{JR}^{\pm} := \langle G_{JR} , B^{\pm}_R \rangle.  \]

For any $i \in I$, the torus $T_R$ normalises $G_{iR}$. The following fact will be useful to us.
\begin{lemma}[{\cite[Lemma 4.3]{MMinject}}]
\label{levi decomp}
Let $A$ be $2$-spherical and suppose $R$ satisfies \eqref{co} in case $A$ is non-spherical. Then for any $i \in I$ and $\varepsilon \in \{\pm\}$,
\[ G_{iR} U^{\varepsilon}_R = U^{\varepsilon}_R G_{iR} \quad \text{ and } \quad G_{iR} B^{\varepsilon}_R = B_R^{\varepsilon} G_{iR}. \]
\end{lemma}
The result below will be essential for our construction of filtrations fitting the requirements of the work in \cite{DevillersMu}.
\begin{lemma}[{\cite[Proposition 4.2]{MMinject}}]
\label{generated by simples 2 spherical}
Let $A=(A_{ij})_{i,j \in I}$ be a $2$-spherical GCM and suppose that 
\begin{enumerate}
\item[($\star$)] For any distinct $i, j \in I$, the group $U^+_{\{i,j\}R}$ is generated by $U_{\alpha_i R}$ and $U_{\alpha_j R}$.
\end{enumerate}
Then $U^{\pm}_R = \langle U_{\pm \alpha_i R} \mid i \in I\rangle$. Moreover, ($\star$) holds as soon as \eqref{co} is satisfied.
\end{lemma}

For any $i \in I$, we define $\tilde{s}_i := x_{\alpha_i}(1) x_{-\alpha_i}(1) x_{\alpha_i}(1)$. One can prove (see \cite[Proposition 7.57]{MarquisKMGbook}) that for any $w \in  \mc{W}$ with reduced expression $(s_{i_1}, \ldots, s_{i_n})$, the element $\tilde{w} := \tilde{s}_{i_1} \ldots \tilde{s}_{i_n} \in G_R$ depends only on $w$ and not on the chosen expression $(s_{i_1}, \ldots, s_{i_n})$. The subgroup $N_R := \langle \tilde{s}_i , T_R \mid i \in I\rangle \leqslant G_R$ surjects onto $\mc{W}$, with kernel $T_R$. Moreover, for any $w \in \mc{W}$ and any $\alpha \in \Phi$, \[\tilde{w} U_{\alpha } \tilde{w}^{-1} =  U_{w   \alpha }. \]

\begin{lemma}
\label{conjugation of UJ+ by J reduced}
Let $A$ be a GCM with Weyl group $\mc{W}$ and set of real roots $\Phi = \Phi^+ \sqcup \Phi^-$. Let $J \subseteq I$ and let $w\in \mc{W}$ be right $J$-reduced. Then 
\[\tilde{w} U_{JR}^+ \tilde{w}^{-1} \subseteq U^+_{R}.\]
\end{lemma}
\begin{proof}
The conjugate $\tilde{w} U^+_{J R}\tilde{w}^{-1}$ is generated by $x_{w \alpha_j}(\lambda)$ for $j \in J$ and $\lambda \in R$. Then by \Cref{length ws implies root positive}, $w \alpha_j \in \Phi^+$ for any $j \in J$, since $w$ is right $J$-reduced. It follows that $\tilde{w} U_{J R}^+ \tilde{w}^{-1} \subseteq U^+_{R}$.  
\end{proof}

An important feature of the theory of Kac--Moody groups over fields is the fact that they have an associated (twin) $BN$-pair, and hence act on an associated (twin) building (see also \cref{sect buildings}). In particular, for $\K$ a field and $\varepsilon \in \{\pm\}$, $G_{\K}$ has \emph{Bruhat decompositions}
\begin{equation}
G_{\K} = \coprod_{w \in \mc{W}} B^{\varepsilon}_{\K} \tilde{w} B^{\varepsilon}_{\K},
\end{equation}
and \emph{Birkhoff decompositions}
\begin{equation}
G_{\K} = \coprod_{w \in \mc{W}} B^{-\varepsilon}_{\K} \tilde{w} B_{\K}^{\varepsilon} = \coprod_{w \in \mc{W}} U^{-\varepsilon}_{\K} \tilde{w} B^{\varepsilon}_{\K},
\end{equation}
see \cite[Corollary 7.70 \textit{\&} §B.3]{MarquisKMGbook}. In particular, from the Bruhat decompositions it follows that
\[\SL_2(\K) = B^{\varepsilon}_2(\K) \sqcup B^{\varepsilon}_{2}(\K) \begin{psmallmatrix}
0 & 1 \\ -1 & 0 
\end{psmallmatrix} B^{\varepsilon}_2(\K),\]
for any field $\K$ and $\varepsilon \in \{\pm\}$.

The double cosets in the Bruhat and Birkhoff decompositions satisfy the following properties (see \cite[Proposition 6.36 \textit{\&} Proposition 6.77]{AbramenkoBrown}). Let $B^{\varepsilon} := B^{\varepsilon}_{\K}$ and $U^{\varepsilon} := U^{\varepsilon}_{\K}$, and write moreover $w$ for $\tilde{w} \in G_{\K}$. Then
\begin{align}
B^{\varepsilon} w B^{\varepsilon} s B^{\varepsilon} & \subseteq B^{\varepsilon} w B^{\varepsilon} \cup B^{\varepsilon} ws B^{\varepsilon}, \quad \text{ and } \notag\\
 B^{\varepsilon} w B^{\varepsilon} s B^{\varepsilon} & = B^{\varepsilon} ws B^{\varepsilon} \iff \ell(ws) = \ell(w)+1, \label{product Bruhat classes}
\end{align}
and 
\begin{align}
U^{-\varepsilon} w B^{\varepsilon} s B^{\varepsilon} & \subseteq U^{-\varepsilon} w B^{\varepsilon} \cup U^{-\varepsilon} ws B^{\varepsilon}, \quad \text{ and } \notag\\
U^{-\varepsilon} w B^{\varepsilon} s B^{\varepsilon} & = U^{-\varepsilon} ws B^{\varepsilon} \iff \ell(ws) = \ell(w)-1, \label{product Birkhoff classes}
\end{align}
One can also show that for $\K$ a field, \[P^{\pm}_{J \K} = B_{\K}^{\pm} \mc{W}_J B_{\K}^{\pm},\] see \cite[§B.3]{MarquisKMGbook}.

Finally, $G_R$ possesses a \emph{Cartan--Chevalley involution} $\omega_R \in \Aut(G_R)$ which preserves $T_R$ and all $\tilde{s}_i$, and which sends $U_{\alpha R}$ to $U_{-\alpha R}$ for all $\alpha \in \Phi$.

\subsubsection*{\texorpdfstring{The groups $\mfG^{\min}_{\mc{D}}(R)$}{The groups Gmin(R)}}
Throughout this section, we assume that $R$ is a local ring. For $A = (A_{ij})_{i,j \in I}$ a GCM, there is also a more involved notion of the \emph{(minimal) Tits functor $\mfG^{\min}_{A}$}, as defined in \cite[Definition 8.126]{MarquisKMGbook}.  
It has the property that for any $i \in I$, there is a group functor morphism $\varphi_i \colon \SL_2 \to \mfG^{\min}_A$, and moreover there is a group functor morphism 
\[\varphi \colon \mfG_A \to \mfG^{\min}_A. \]
We write $\varphi_{iR} \colon \SL_2(R) \to \mfG^{\min}_A(R)$ and $\varphi_R \colon \mfG_A(R) \to \mfG^{\min}_A(R)$ for the respective group homomorphisms naturally induced by $\varphi_i$ and $\varphi$. Then for any $r \in R$, and any $i \in I$, the map $\varphi_R$ has the property that
\[ \varphi_{iR}\begin{psmallmatrix} 1 & r \\ 0 & 1 \end{psmallmatrix} = \varphi_R  (x_{\alpha_i}(r)) \quad \text{ and } \quad \varphi_{iR}\begin{psmallmatrix} 1 & 0 \\ -r & 1\end{psmallmatrix} = \varphi_{R}(x_{-\alpha_i}(r)). \]

Additionally, for any injective ring morphism $R_1 \to R_2$, the induced group homomorphism $\mfG_{A}^{\min}(R_1) \to \mfG_{A}^{\min}(R_2)$ is injective. 

Since $R$ is a local ring, and hence $\SL_2(R) = \langle U^+_2(R), U^-_2(R)\rangle$, the group $\mfG^{\min}_A(R)$ coincides with $\varphi_R(\mfG_{A}(R))$, and the work in \cite{MMinject} is dedicated to investigating when $\varphi_R$ is injective, i.e.\ to when the groups $\mfG_{A}(R)$ and $\mfG_{A}^{\min}(R)$ coincide. This was already known by Tits' work in \cite{Tits87} to be the case when $R = \K$ is a field, and is moreover proven to be the case when $R$ satisfies (\ref{co}) and is a \emph{valuation ring}, i.e.\ a local ring which is moreover a Bézout domain in \cite[Theorem A]{MMinject}.

When $R$ is a domain with field of fractions $\K$, there is a natural map \[\psi \colon \mfG_{A}(R) \to \mfG_{A}(\K),\] induced by the canonical inclusions $\mfU_{\alpha}(R) \into \mfU_{\alpha}(\K)$, for all $\alpha \in \Phi$.

Since $\mfG_{A}(\K) \cong \mfG_{A}^{\min}(\K)$, the group $\mfG^{\min}_{A}(R)$ injects into $\mfG_{A}(\K)$, and under this identification, the map \[\varphi_R \colon \mfG_{A}(R) \to \mfG_{A}^{\min}(R) \subseteq \mfG_{A}(\K)\] coincides with $\psi$, so that $\mfG_{A}^{\min}(R)$ is just the image of $\mfG_{A}(R)$ in $\mfG_{A}(\K)$ induced by the inclusion $R \into \K$.

Throughout the rest of this article, we will write $G^{\min}_R := \mfG_{A}^{\min}(R)$. Moreover, for a subgroup $H \leqslant G_R$, we will denote by $\ov{H}$ its image under $\varphi_R$ in $G^{\min}_R$, e.g.\ $\ov{U}^{\pm}_R := \varphi_R(U^{\pm}_R)$ and 
\[\ov{G}_{iR} := \varphi_R(G_{iR}) \cong  \left\langle \begin{psmallmatrix} 1 & r \\ 0 & 1 \end{psmallmatrix}, \begin{psmallmatrix} 1 & 0 \\ -r & 1\end{psmallmatrix} \mid r \in (R,+)\right\rangle \cong \SL_2(R),\]
where the latter isomorphism follows because $R$ is a local ring.

The restriction of $\varphi_R$ to any root subgroup $U_{\alpha}$ is injective. We again write $T_R$ $, \tilde{s}_i$ and $U_{\alpha}$ for their respective images under $\varphi_R$. In particular, $\tilde{s}_i = \varphi_{iR}\begin{psmallmatrix} 0 & -1 \\ -1 & 0\end{psmallmatrix}$, for all $i \in I$.

\subsection{\texorpdfstring{Some facts about $G_R$ and $G^{\min}_R$}{Some facts about G(R) and Gmin(R)}}

Let $R$ be a local ring and denote by $R \to k$ the projection onto its residue field $k := R/L$, where $L$ is the maximal ideal of $R$. We then write 
\begin{equation}
\label{definition projection onto residue field}
p \colon G^{\min}_R \to G^{\min}_k \cong G_k
\end{equation}
for the group homomorphism induced by functoriality.

If we denote respectively by $\ov{U}^{\pm}_L$ and $\ov{B}^{\pm}_L$ the kernels of \[p\hspace{-1pt}\mid_{\ov{U}^{\pm}_R} \colon \ov{U}^{\pm}_R \to U^{\pm}_k \quad \text{  and } \quad p\hspace{-1pt} \mid_{\ov{B}^{\pm}_R} \colon \ov{B}^{\pm}_R \to B^{\pm}_k,\] one has the following.
\begin{lemma}
\label{Kernel p}
Let $R$ and $p \colon G^{\min}_R \to G_k$ be as above. Then
\begin{equation}
\Ker(p) = \ov{U}^-_L \ov{B}^+_L = \ov{U}^+_L \ov{B}^-_L.
\end{equation}
\end{lemma}
\begin{proof}
By \cite[Proposition 5.3]{MMinject} and surjectivity of $\varphi_R$, one has that $\Ker(p) \subseteq \ov{U}_L^- \ov{B}^+_R$. If $g = u_- b_+ \in \Ker(p)$ with $u_-\in \ov{U}^-_L$ and $b_+ \in \ov{B}^+_R$, then $u^{-1}_- g \in \Ker(p)$ and hence $b_+ \in \ov{B}^+_L$. The result now follows from the fact that $\Ker(p)^{-1} = \Ker(p)$ and that $T_R$ normalises $\ov{U}^+_L$.
\end{proof}

The following lemma is particularly useful.

\begin{lemma}[{\cite[Lemma 5.10]{MMinject}}]
\label{5.10}
Let $A = (A_{ij})_{i,j \in I}$ be $2$-spherical, and let $R$ be a local ring satisfying \eqref{co} in case $A$ is not spherical. Let $\pi_R\colon G_R \to G_k$ denote the composition of $\varphi_R \colon G_R \to G^{\min}_R$ and $p \colon G^{\min}_R \to G_k$. Then, for $J \subseteq I$, \[G_{JR} \cap \pi^{-1}_R(B^-_k B^+_k) \subseteq T_R U_{JR}^- U_{JR}^+.\]
\end{lemma}

We also record the following lemma for convenience. Here $U^{\pm}_2(L)$ denotes the kernel of the canonical map $\SL_2(R) \to \SL_2(k)$ induced by $R \to k$, restricted to $U^{\pm}_2(R)$. 
\begin{lemma}[{\cite[Lemma 5.2]{MMinject}}]
\label{lemma preimage B+2}
Let $i \in I$ and $g \in \ov{G}_{iR}$. If $p(g) \in \varphi_{ik}(B^+_2(k))$, then $g \in \varphi_{iR}(U^-_2(L) B^+_2(R))$. Similarly, if $p(g) \in \varphi_{ik}(B^-_2(k))$, then $g \in \varphi_{iR}(U_2^+(L) B^-_2(R))$. 
\end{lemma}

When $R$ is a domain, under the identification $G_{\K}^{\min} \cong G_{\K}$, we have \begin{equation}
\label{Bir}
\ov{U}^{\pm}_R \cap \ov{B}^{\mp}_R \subseteq U^{\pm}_{\K} \cap B^{\mp}_{\K} = \{1\},
\end{equation}
where the last equality follows from the Birkhoff decomposition of $G_{\K}$ (see \cite[eq. (2.1)]{MMinject}). This condition also implies that $T_R \cap \ov{U}^{\pm}_R = \{1\}$.
\begin{lemma}
Let $A = (A_{ij})_{i,j \in I}$ be a $2$-spherical GCM. Let $R$ be a local domain satisfying \eqref{co}. Then for any $J \subseteq I$, 
\begin{equation}
\label{intersection B+ P-}
\ov{U}^{+}_R \cap \ov{P}^{-}_{JR} = \ov{U}^{+}_{JR}, 
\end{equation}
\end{lemma}
\begin{proof}
The inclusion $\ov{U}^+_{JR} \subseteq  \ov{U}^+_R \cap \ov{P}^-_{JR}$ is trivial. For the opposite inclusion, note first that 
\[\ov{U}^+_R \cap \ov{P}^-_{JR} = \ov{U}^+_R \cap \ov{P}^+_{JR} \cap \ov{P}^-_{JR} = \ov{U}^+_R \cap T_R \ov{G}_{JR}, \]
where the latter equality follows from \cite[Proposition 5.12]{MMinject}. Suppose that $g \in U^+_R$ satisfies $\varphi_R(g) \in T_R \ov{G}_{JR}$. Let $t \in T_R$ and $g_J \in G_{JR}$ be such that $\varphi_{R}(g) = \varphi_R(t g_J)$.
Then \[p(\varphi_R(g_J)) = p(\varphi_R(t^{-1}g)) \in B^+_k \subseteq B^-_k B^+_k.\]
In particular, by \Cref{5.10} it follows that $t g_J \in T_R U^-_{JR} U^+_{JR}$. Choose $t' \in T_R$ and $u_{\pm} \in U^{\pm}_{JR}$ such that $tg_J = t' u_- u_+$. In particular, $\varphi_R(g) = \varphi_R(t' u_- u_+)$. Then, by \eqref{Bir},
\[\varphi_R(t'^{-1} g u_+^{-1}) = \varphi_R(u_-) \in T_R \ov{U}^+_R \cap \ov{U}^-_{JR} = \{1\}. \] 
Hence $\varphi_R(t') = \varphi_R(g u_+^{-1}) \in T_R \cap \ov{U}^+_R = \{1\}$,
and we obtain that $\varphi_R(g) = \varphi_R(u_+) \in \ov{U}^+_{JR}$. 
\end{proof}

\section{\texorpdfstring{Birkhoff decomposition of $G^{\min}_R$}{Birkhoff decomposition of Gmin(R)}}
\label{sect birkhoff}
Throughout this section, we fix a $2$-spherical GCM $A = (A_{ij})_{i,j\in I}$, and $R$ a local ring. We aim to construct a decomposition of $G^{\min}_R$ which generalises the Birkhoff decompositions in the field case.

\begin{definition}
Let $g \in G^{\min}_R$. A \emph{Birkhoff expression of $g$} is a decomposition of the form 
\[ g = u_+ g_1 \ldots g_d b_-, \]
with $u_+ \in \ov{U}^+_R$ and $b_- \in \ov{B}^-_R$, and such that for all $1 \leq j \leq d$, there is some $i_j \in I$ such that $g_j \in \ov{G}_{i_j R}$. The \emph{type} of such a Birkhoff expression of $g$ is the tuple $(s_{i_1}, \ldots, s_{i_d})$, for $s_i \in \mc{W}$. The Birkhoff expression $g = u_+ g_1 \ldots g_d b_+$ is \emph{reduced} if $d$ is minimal among the lengths of types of Birkhoff expressions of $g$. 
\end{definition}

Let $(m_{ij})_{i,j \in I}$ denote the Coxeter matrix of $A$ and for any $i,j \in I$, let $\mathbf{t}(i,j)$ and $\tau(i,j)$ be as in \cref{sect coxeter}. The following result is a construction of subsets of $G^{\min}_R$ generalising the double cosets appearing in the Birkhoff decomposition of a Kac--Moody group over a field.
\begin{proposition}
\label{uniqueness birkhoff classes}
Let $R$ be a local ring satisfying \eqref{co} in case $A$ is not spherical. 
Let $p \colon G_R^{\min} \to G_k$ be as in \eqref{definition projection onto residue field} and let $\mathbf{w} = (s_{i_1},\ldots, s_{i_d})$ be a word in the generators $\{s_i \mid i \in I\}$ of $\mc{W}$, and let $w := s_{i_1} \ldots s_{i_d} \in \mc{W}$.  Define
\begin{align*}
B(\mathbf{w}) := \Big\{ g \in G^{\min}_R \mid g \text{ has a reduced Birkhoff expression of type } (s_{i_1}, \ldots, s_{i_d})\Big\}.
\end{align*}
Then the following hold:
\begin{enumerate}[(i)]
\item $g \in B(\mathbf{w})$ implies that $p(g) \in U^+_k \tilde{w} B^-_k$. \label{birk1}
\item Let $\mathbf{w'} = (s_{j_1}, \ldots, s_{j_d})$ be another word of length $d$ representing $w$. Then $B(\mathbf{w})~=~ B(\mathbf{w'})$. \label{birk3} 
\item The set $B(\mathbf{w})$ is non-empty if and only if $\mathbf{w}$ is reduced. \label{birk2}
\end{enumerate}
\end{proposition}
\begin{proof}
\eqref{birk1}. Let $\mathbf{w} = (s_{i_1}, \ldots, s_{i_d})$. We prove the statement by induction on $d$. Let $g \in  B(\mathbf{w})$. Suppose $d = 0$. Then $g \in  \ov{U}^+_R\ov{B}^-_R$. It then follows by definition that $p(g) \in U^+_k B^-_k$.

Suppose now that for any $\mathbf{w} = (s_{i_1}, \ldots, s_{i_d})$ with $d \leq n$ and any $g \in B(\mathbf{w})$, one has $p(g) \in U^+_k \tilde{w} B^-_k$, where $w = s_{i_1} \ldots s_{i_d} \in \mc{W}$.

Let then $\mathbf{w} = (s_{i_1}, \ldots, s_{i_{n+1}})$ and $g \in B(\mathbf{w})$, i.e.\ \[g = u_+ g_1 \ldots g_{n+1} b_-, \text{ for some } u_+ \in \ov{B}^-_R, \quad  g_j \in \ov{G}_{i_jR}, \quad \text{ and } b_-\in \ov{B}^-_R,\]
with $n+1$ minimal amongst the lengths of Birkhoff expressions of $g$. By \Cref{levi decomp}, $\ov{B}^-_R \ov{G}_{iR} = \ov{G}_{iR} \ov{B}^-_R $ and $\ov{G}_{iR} \ov{U}^+_R = \ov{U}^+_R \ov{G}_{iR}$, for any $i \in I$. Hence for all $j \in \{1, \ldots, n+1\}$, \[g_j \not \in \varphi_{i_j R}(U^+_2(R)B^-_2(R)) \subseteq \ov{U}^+_R \ov{B}^-_R,\] by minimality of $n+1$. Then \Cref{lemma preimage B+2} implies that $p(g_j) \not\in  \varphi_{i_j k }(B^-_2(k))$. It follows in particular that 
\[p(g_{n+1}) \in \varphi_{i_{n+1}k}\Big(B^-_2(k) \begin{psmallmatrix} 0 & 1 \\ -1 & 0\end{psmallmatrix} B^-_2(k)\Big) \subseteq B^-_k \tilde{s}_{i_{n+1}} B^-_k.\]
Additionally, writing $\mathbf{w}' := (s_{i_1}, \ldots, s_{i_n})$, we have
\[g' := u_+ g_1 \ldots g_n \in B(\mathbf{w}'),\] 
since the above decomposition not being a reduced Birkhoff expression of $g'$, would contradict the reducedness of the Birkhoff expression $u_+g_1 \ldots g_{n+1} b_-$ of $g = u_+ g_1 \ldots g_{n+1} b_- = g' g_{n+1} b_-$. By the induction hypothesis, 
\[p(g') \in U^+_k \tilde{s}_{i_1} \ldots \tilde{s}_{i_n} B^-_k.\]
Combining the above then implies by \eqref{product Birkhoff classes} that 
\[p(g) = p(g' g_{n+1} b_-) \in U^+_k \tilde{s}_{i_1} \ldots \tilde{s}_{i_n} B^-_k \tilde{s}_{i_{n+1}} B^-_k \subseteq U^+_k \{ \tilde{s}_{i_1} \ldots \tilde{s}_{i_{n}}, \tilde{s}_{i_1} \ldots \tilde{s}_{i_{n+1}}\} B^-_k.\]
Suppose that $p(g) \in U^+_k \tilde{s}_{i_1} \ldots \tilde{s}_{i_n} B^-_k$. Then by \Cref{Kernel p}, 
\begin{equation*}
g  \in \ov{U}^+_R \tilde{s}_{i_1} \ldots \tilde{s}_{i_n} \Ker(p) \ov{B}^-_R  \subseteq \ov{U}^+_R \tilde{s}_{i_1} \ldots \tilde{s}_{i_n} \ov{U}^+_L \ov{B}^-_R \subseteq \ov{U}^+_R \ov{G}_{i_1 R} \ldots \ov{G}_{i_n R} \ov{B}^-_R,
\end{equation*}
where the last inclusion follows from \cite[Lemma 5.1]{MMinject}, which implies that $\ov{G}_{i R} \ov{U}^+_L = \ov{U}^+_L \ov{G}_{i R}$ for any $i\in I$. But this is a contradiction with the minimality of $n+1$, and hence \[p(g) \in U^+_k \tilde{s}_{i_1} \ldots \tilde{s}_{i_{n+1}} B^-_k,\]
as desired. This concludes the proof of \eqref{birk1}.

\eqref{birk3}. Suppose $\mathbf{w}' = (s_{j_1}, \ldots, s_{j_d})$ represents $w = s_{i_1} \ldots s_{i_d} = s_{j_1} \ldots s_{j_d} \in \mc{W}$. Let $g \in B(\mathbf{w})$. By \eqref{birk1}, \[p(g) \in U^-_k \tilde{s}_{i_1} \ldots \tilde{s}_{i_d} B^+_k = U^-_k \tilde{w} B^+_k = U^-_k \tilde{s}_{j_1} \ldots \tilde{s}_{j_d} B^-_k.\] In particular, since $\ov{U}^+_R \tilde{w} \ov{B}^-_R$ surjects onto $U^+_k \tilde{w} B^-_k$, we have as in the previous paragraph that 
\[g \in \ov{U}^+_R \tilde{s}_{j_1} \ldots \tilde{s}_{j_d} \Ker(p) \ov{B}^-_R \subseteq \ov{U}^+_R \tilde{s}_{j_1} \ldots \tilde{s}_{j_d} \ov{U}^+_L\ov{B}^-_R \subseteq \ov{U}^+_R \ov{G}_{j_1} \ldots \ov{G}_{j_d} \ov{B}^-_R,\]
by \Cref{Kernel p} and \cite[Lemma 5.1]{MMinject}. Hence $g$ has a Birkhoff expression of type $(s_{j_1}, \ldots, s_{j_d})$. Since $d$ is the length of a reduced Birkhoff expression of $g$, it follows that this Birkhoff expression is reduced as well. Thus $g \in B(\mathbf{w}')$, yielding \eqref{birk3}.

\eqref{birk2}. Suppose $\mathbf{w} = (s_{i_1}, \ldots, s_{i_d})$ is not reduced. Suppose for a contradiction that there exists $g \in B(\mathbf{w})$. By Tits' solution to the word problem, there is an integer $n \geq 1$ such that, after applying some $n$ operations of type (\hyperlink{M2}{M2}), the obtained word $\mathbf{w'} = (s_{j_1}, \ldots, s_{j_d})$ has a subword of the form $(s_{i}, s_i)$. As $B(\mathbf{w}) = B(\mathbf{w}')$ by \eqref{birk3}, it follows that \[g \in \ov{U}^-_R \ov{G}_{j_1 R}\ldots \ov{G}_{iR} \ov{G}_{iR} \ldots \ov{G}_{j_d R} \ov{B}^-_R = \ov{U}^+_R \ov{G}_{j_1 R} \ldots \ov{G}_{i R} \ldots \ov{G}_{j_d R},\]
where the equality holds because $\ov{G}_{iR} \ov{G}_{iR} = \ov{G}_{iR}$. But this would mean that $g$ has a Birkhoff expression of length $d-1$, contradicting the reducedness of the expression of $g$ of type $\mathbf{w}$. Hence $B(\mathbf{w})$ is empty if $\mathbf{w}$ is not reduced. 

Conversely, if $\mathbf{w} = (s_{i_1}, \ldots s_{i_d})$ is reduced, then $g := \tilde{s}_{i_1} \ldots \tilde{s}_{i_n}$ is a reduced Birkhoff expression of $g$. Indeed, by definition, $p(\tilde{s}_{i_1} \ldots \tilde{s}_{i_n}) \in U^-_k \tilde{w} B^+_k$ with $w = s_{i_1} \ldots s_{i_n}$. Since $\mathbf{w}$ is reduced, $\ell(w) = n$. If $g$ had a Birkhoff expression of shorter length, we would have by \eqref{birk1} that $p(g) \in U^+_k w' B^-_k$ with $\ell(w') < \ell(w)$, a contradiction with the Birkhoff decomposition of $G_k$. Hence $g \in B(\mathbf{w})$.
\end{proof}
It follows from \Cref{uniqueness birkhoff classes} that for any $w\in \mc{W}$, the set $B(\mathbf{w})$ does not depend on the chosen reduced expression $\mathbf{w}$ of $w$. We then write 
\begin{equation}
\label{def Bw}
B(w) := B(\mathbf{w}), \quad \text{ for any reduced expression } \mathbf{w} = (s_{i_1}, \ldots, s_{i_d}) \text{ of } w \in \mc{W}.
\end{equation}

\begin{theorem}
\label{thm birkhoff}
Let $R$ be a local ring satisfying \eqref{co} in case $A$ is not spherical. Then $G^{\min}_R$ has a decomposition
\begin{equation}
\label{eq Birkhoff local}
G^{\min}_R = \coprod_{w \in \mc{W}} B(w).
\end{equation}
Moreover,
\begin{equation}
\label{inverse image Birkhoff classes}
B(w) = p^{-1}(U^+_k \tilde{w} B^-_k) = \ov{U}^+_R \tilde{w} \ov{B}^-_R \Ker(p).
\end{equation}
\end{theorem}
\begin{proof}
Recall that $G^{\min}_R$ is generated by the groups $\ov{G}_{iR}$ for $i \in I$. Thus for every element $g \in G^{\min}_R$ there is some sequence $(i_1, \ldots, i_d) \in I^d$ such that
\[ g \in \ov{G}_{i_1 R} \ldots \ov{G}_{i_d R},\]
and hence $g$ has a Birkhoff decomposition. In particular, $g \in B(w)$ for some $w \in \mc{W}$.
It follows that $G^{\min}_R = \bigcup_{w \in \mc{W}} B(w)$.

Let $w \neq w' \in \mc{W}$ and suppose $g \in B(w) \cap B(w')$. Then by \Cref{uniqueness birkhoff classes}\eqref{birk1}, \[p(g) \in U^+_k \tilde{w} B^-_k \cap U^+_k \tilde{w}' B^-_k,\]
a contradiction by the Birkhoff decomposition of $G_k$. Hence $B(w) \cap B(w') = \emptyset$. The decomposition \eqref{eq Birkhoff local} follows. By \Cref{uniqueness birkhoff classes}\eqref{birk1}, $B(w)$ surjects onto $U^-_k \tilde{w} B^+_k$, and hence $B(w) \subseteq p^{-1}(U^-_k \tilde{w} B^+_k)$. Conversely, if $g \not\in B(w)$, then by \eqref{eq Birkhoff local}, $g \in B(w')$ for some $w' \neq w$. Then again by \Cref{uniqueness birkhoff classes}\eqref{birk1}, $p(g) \in U^-_k \tilde{w}' B^+_k$, and hence $p(g) \not \in U^-_k \tilde{w} B^+_k$, by the Birkhoff decomposition of $G_k$. In particular, $g \not\in p^{-1}(U^-_k \tilde{w} B^+_k)$. This shows \eqref{inverse image Birkhoff classes}. 
\end{proof}

\begin{remark}
Note that using the Cartan--Chevalley involution, $G^{\min}_R$ also has a decomposition  $G^{\min}_R = \coprod_{w\in \mc{W}} \tilde{B}(w)$, where, for $w \in \mc{W}$,
\begin{equation*}
\tilde{B}(w) := p^{-1}(U^-_k \tilde{w} B^+_k) = \ov{U}^-_R \tilde{w} \ov{B}^+_R \Ker(p) .
\end{equation*}
\end{remark}

\section{\texorpdfstring{Residual filtrations of the chamber system of $G^{\min}_R$}{Residual filtrations of the chamber system of Gmin(R)}}
\label{sect filtration}
In this section, we consider the chamber system associated to $G^{\min}_R$. In particular, we will construct a \emph{filtration} of this chamber system, and use the results from \cite{DevillersMu} to reduce the question of obtaining a presentation of an arbitrary $\ov{U}^+_R$ of simply laced type to obtaining a presentation for $U^+_R$ of type $A_3$ (in case $R$ is a local Bézout domain). Throughout this section, we assume as before that $R$ is a local ring with residue field $k$.
\subsection{Buildings and chamber systems}
\label{sect buildings}
A \emph{chamber system} is a set $C$ together with a collection $(\sim_i~)_{i \in I}$ of equivalence relations on $C$, such that $c \sim_i d$ and $c \sim_j d$ if and only if $i = j$ or $c = d$. Elements of $C$ are called \emph{chambers}, and two chambers $c,d \in C$ are called \emph{$i$-adjacent} if $c\sim_i d$. A chamber system is said to be \emph{connected} if any two chambers $c,d \in C$ can be joined by a \emph{gallery} \[c = c_1 \sim_{i_1} c_2 \sim_{i_2} \ldots \sim_{i_{k}} c_{k+1} = d.\] The \emph{type} of such a gallery is the tuple $(i_1, \ldots, i_k)$.  Let $J\subseteq I$. A \emph{$J$-gallery} is a gallery $c~\sim_{j_1}~\ldots~\sim_{j_k}d$ such that $j_1, \ldots , j_k \in J$. For $c \in C$, the \emph{$J$-residue} of $c$ is the set $\mc{R}_J(c)$ of chambers $d \in C$ which can be joined to $c$ by a $J$-gallery. Then $(\mc{R}_J(c), (\sim_j)_{j \in J})$ is also a chamber system. 
\begin{example}
\label{example chamber system structure}
Let $R$ satisfy \eqref{co} and let $A = (A_{ij})_{i,j\in I}$ be $2$-spherical. Then we associate a chamber system $\mc{C}^-(G^{\min}_R)$ to $G^{\min}_R$ by taking the coset space $G^{\min}_R/\ov{B}^-_R$ as the set of chambers, and for any $i \in I$, letting $g\ov{B}^-_R$ be $i$-adjacent to $h \ov{B}^-_R$ if and only if $h^{-1} g \in \ov{P}^-_{iR}$. In particular, since $\ov{P}^-_{iR} = \ov{G}_{iR} \ov{B}^-_R$ by \Cref{levi decomp},
\[g\ov{B}^-_R \sim_i h\ov{B}^-_R \iff \text{ there exists } \quad g_i \in \ov{G}_{iR} \quad \text{ such that }  \quad h\ov{B}^-_R = g g_{i} \ov{B}^-_R.\] We will write 
\begin{equation}
g c_0^- := g \ov{B}^-_R \quad \text{for the chamber in } \mc{C}^-(G^{\min}_R) \text{ given by the coset } g\ov{B}^-_R\in G^{\min}_R/\ov{B}^-_R.
\end{equation}
The group $G^{\min}_R$ acts on its chamber system by left multiplication, and this action preserves $i$-adjacency for any $i \in I$. In other words, $G^{\min}_R$ acts on $\mc{C}^-(G^{\min}_R)$ by chamber system automorphisms. Stabilisers of $J$-residues of the chamber $c_0^-$ are given by the parabolic subgroups $\ov{P}_{JR}^-$. 
If $R = \K$ is a field, the chamber system $\mc{C}^-(G_{\K}) = \mc{C}^-(G^{\min}_{\K})$ is moreover a building (see \cite[Corollary~7.70]{MarquisKMGbook}).
\end{example}

A chamber system is said to be \emph{simply connected} (or more precisely, \emph{simply $2$-connected}) if every gallery which starts and ends in the same chamber $c$, is \emph{$2$-homotopic} to the gallery $(c)$. Here a $2$-homotopy consists of a finite sequence of \emph{elementary $2$-homotopies}, which in turn consist of replacing a $J$-subgallery (with $\abs{J} \leq 2$) connecting two chambers by another $J$-subgallery connecting the same two chambers (see \cite[§2]{DevillersMu} for the precise definition).

An important feature of the theory of groups acting on chamber systems, is the following. 
\begin{proposition}[{\cite[Corollaire  1]{Tits86}}]
\label{Tits presentation}
Let $\mc{C}:= (C, (\sim_i)_{i \in I})$ be a chamber system, and let $G$ be a group acting transitively on $C$ such that the action preserves $i$-adjacency. Let $c \in C$. Then the following are equivalent:
\begin{enumerate}[(a)]
\item $\mc{C}$ is simply connected;
\item $G$ is the amalgamated product of the stabilisers \[G_J := \mathrm{Stab}_{G}(\mc{R}_J(c)) = \{g \in G \mid g\cdot c' \in \mc{R}_J(c),  \text{ for all } c' \in \mc{R}_J(c)\},\] for all $J\subseteq I$ with $\abs{J} \leq 2$, along the canonical inclusions  $G_{J_1} \subseteq G_{J_2}$ for $J_{1} \subseteq J_2\subseteq I$. 
\end{enumerate}
\end{proposition}
See e.g.\ \cite[§1]{SerreTrees} for a precise definition of such an amalgamated product. To obtain a presentation of a group $G$ as an amalgamated product of a certain set of subgroups, it hence suffices to show that $G$ acts on a simply connected chamber system and to identify the stabilisers of the rank $2$ residues. To obtain chamber systems suited to our purposes, we need the following terminology.

\begin{definition}[{\cite[§3]{DevillersMu}}]
\label{def filt}
Let $\mc{C} = (C,(\sim_i)_{i \in I})$ be a chamber system. A \emph{filtration of $\mc{C}$} is a sequence $\mc{F} := (C_n)_{n \in \N}$ of subsets of $C$, satisfying the following conditions:
\begin{enumerate}[(i)]
\item $C_{n} \subseteq C_{n+1}$ for all $n \in \N$,
\item $\bigcup_{n\in \N} C_n = C$,
\item For every $n > 0$, if $C_{n-1} \neq\emptyset$, then there exists an $i \in I$ such that for all $c \in C_n$, there is some $c' \in C_{n-1}$ which is $i$-adjacent to $c$.
\end{enumerate}
The filtration $\mc{F}$ is called \emph{residual} if for every non-empty $J\subseteq I$ and any $J$-residue $\mc{R}_J$ in $\mc{C}$, the sequence $(C_n \cap \mc{R}_J)_{n \in \N}$ is a filtration of the chamber system $(\mc{R}_J, (\sim_{j})_{j \in J}))$. For any subset $X \subseteq C$, we define 
\begin{equation}
\aff(X) :=  C_{m_X} \cap X, \quad \text{ where } m_X := \min\{n \in \N \mid C_n \cap X \neq \emptyset\}.
\end{equation}
A residual filtration $\mc{F}$ is said to be \emph{locally connected} if it satisfies the following condition:
\begin{align}
\label{lco}
\tag{lco}  & \text{For any } J \subseteq I \text{ with } \abs{J} \leq 2, \text{ and any } J\text{-residue } \mc{R}_J, \text{ the chamber system }\\
 & (\aff(\mc{R}_J),(\sim_j)_{j \in J}) \text{ is connected}.\notag
\end{align}
Finally, $\mc{F}$ is said to be \emph{locally simply connected} if it satisfies:
\begin{align}
\label{lsco}
\tag{lsco}  & \text{For any } J \subseteq I \text{ with } \abs{J} \leq 3, \text{ and any } J\text{-residue } \mc{R}_J, \text{ the chamber system } \\
& (\aff(\mc{R}_J),(\sim_j)_{j \in J}) \text{ is simply connected}.\notag
\end{align}
\end{definition}
We will make crucial use of the following three facts.
\begin{proposition}[{\cite[Corollary 3.10 {\it \&} Theorem 3.14]{DevillersMu}}]
\label{314 devmu}
Let $\mc{F} = (C_n)_{n \in \N}$ be a residual filtration of a chamber system $\mc{C}$, with $C_0 \neq \emptyset$ and satisfying \eqref{lco} and \eqref{lsco}. Then the following are equivalent:
\begin{enumerate}[(i)]
\item $\mc{C}$ is simply connected.
\item $C_0$ is simply connected.
\end{enumerate}
\end{proposition}

\begin{proposition}[{\cite[Theorem 4.3]{Ronan}}] 
\label{buildings are sc}
Buildings are simply connected chamber systems.
\end{proposition}

\begin{proposition}[{\cite[Lemma 8.7]{MMinject}}]
\label{coincides w building of K}
Let $R$ be a local Bézout domain with field of fractions $\K$, satisfying \eqref{co} in case $A$ is not spherical. Then the canonical map 
\[\mc{C}^-(G^{\min}_R) \to \mc{C}^-(G_{\K}), \quad g \ov{B}^-_R \to g B_{\K}^-\]
is an isomorphism of chamber systems.
\end{proposition}
In particular, in view of \Cref{buildings are sc} and the fact that $\mc{C}^-(G_{\K})$ is a building by \Cref{example chamber system structure}, under the hypotheses of \Cref{coincides w building of K}, $\mc{C}^-(G^{\min}_R)$ is simply connected. In the following section, we will construct a suitable filtration $(C_n)_{n \in \N}$ of $\mc{C}^-(G^{\min}_R)$ such that $\ov{U}^+_R$ acts transitively on $C_0$, and the latter will be simply connected by \Cref{314 devmu}. We will then apply \Cref{Tits presentation} to obtain \Cref{main theorem}.

\subsection{\texorpdfstring{Filtration of $\mc{C}^-(G^{\min}_R)$}{Filtration of C(GminR)}}

We define the map 
\begin{equation}
\delta \colon \mc{C}^-(G^{\min}_R) \to \mc{W}, \quad \text{ via } \quad \delta(g c_0^-) = w \iff g \in B(w).
\end{equation}
This is well-defined by \eqref{eq Birkhoff local} and the definition of $B(w)$ (see \eqref{def Bw}). As in \cref{sect prelim}, let $\ell$ denote the length function on $\mc{W}$, with respect to $\{s_i\mid i \in I\}$. We fix an injection $\abs{\cdot} \colon \mc{W} \to \N$ satisfying: 
\begin{enumerate}[(i)]
\item $\abs{1_{\mc{W}}} = 0$,
\item $\ell(w_1) < \ell(w_2) \implies \abs{w_1} < \abs{w_2}$.
\end{enumerate}
Let for any $n \in \N$, 
\begin{equation}
\label{definition filtration}
C_n := \left\{gc_0^-\in \mc{C}^-(G^{\min}_{R})  \mid \abs{\delta(gc_0^-)} \leq n\right\}.
\end{equation}

\begin{lemma}
\label{lemma gate properties}
Let $J \subseteq I$ and let $\mc{R}_J$ be a $J$-residue in $\mc{C}^-(G^{\min}_R)$. Let $gc_0^- \in \mc{R}_J$ and set $ w = \delta(gc_0^-)\in \mc{W}$. Set moreover \[l := \min \left\{\ell(\delta(h c_0^-)) \mid hc_0^-\in \mc{R}_J\right\}.\] Then the following hold:
\begin{enumerate}[(a)]
\item For any $h c_0^- \in \mc{R}_J$, one has $\delta(hc_0^-) \in w \mc{W}_J$. \label{lemma 1 item a} In particular,
\begin{equation}
\label{description residues}
\left\{\delta(h c_0^-) \mid hc_0^- \in \mc{R}_J \right\} = w \mc{W}_J.
\end{equation}

\item \label{lemma 1 item b} $\aff(\mc{R}_J) = \left\{hc_0^- \in \mc{R}_J \mid \ell(\delta(hc_0^-)) = l \right\} 
 =  \left\{h c_0^- \in \mc{R}_J \mid \delta(hc_0^-) = w_J \right\}$,
where $w_J$ is the unique element of $w \mc{W}_J$ of minimal length.
\end{enumerate}
\end{lemma}
\begin{proof}
Let $g c_0^-$ be as in the statement. Let moreover $h c_0^- \in \mc{R}_J$. Let $(i_1, \ldots, i_d)$ be the type of a $J$-gallery from $gc_0^-$ to $hc_0^-$. Then for $1 \leq t \leq d$, there exist $g_{t} \in \ov{G}_{i_{t}R}$ such that \[g c_0^- \sim_{i_1} g g_1 c_0^- \sim_{i_2}  \ldots \sim_{i_d} g g_1  \ldots g_d c_0^- = hc_0^-.\] 
Now, $p(g) \in U^+_k \tilde{w} B^-_k$ by \Cref{uniqueness birkhoff classes}. Moreover, since $g_t \in \ov{G}_{i_t R} $ for all $1 \leq t \leq d$, we have $p(g_t) \in B^-_k \sqcup B^-_{k} \tilde{s}_{i_t} B^-_k$. Hence by \eqref{product Bruhat classes}, there is some subword $(s_{j_1}, \ldots, s_{j_n})$ of $(s_{i_1}, \ldots ,s_{i_d})\in \{s_j \mid j \in J\}^d$ such that \[p(g_1\ldots g_d) \in B^-_k \tilde{s}_{j_1} \ldots \tilde{s}_{j_n} B^-_k.\]
It follows that 
\[p(h) \in U^+_k\tilde{w} B^-_k \tilde{s}_{j_1} \ldots \tilde{s}_{j_n} B^-_k, \]
and hence by \eqref{product Birkhoff classes}, there is some subword $\mathbf{w'}$ of  $(s_{j_1}, \ldots , s_{j_n})$ such that for the associated element $w' \in \mc{W}_J$, we have $p(h) \in U^+_k \tilde{w} \tilde{w}' B^-_k$. By \eqref{inverse image Birkhoff classes} it now follows that $h \in B(w w')$. In particular, $\delta(h c_0^-) \in w \mc{W}_J$. This shows one inclusion of \eqref{description residues}.

Conversely, note that every element which is of the form $g x c_0^-$ for some $x \in \ov{P}^-_{JR}$ belongs to $\mc{R}_J$. Since $p(g) \in U^+_k \tilde{w} B^-_k$, there exist $u_+ \in \ov{U}^+_R$ and $b_- \in \ov{B}^-_R$ such that $p(g) = p(u_+ \tilde{w} b_-)$. Then, for any $w' \in \mc{W}_J$, setting $x := b_-^{-1}\tilde{w}' \in \ov{P}^-_{JR}$, we have that 
\[p(g)p(x) = p(u_+ \tilde{w} \tilde{w}') \in U^+_k \tilde{w} \tilde{w}' B^-_k.\]
In other words, for any $w' \in \mc{W}_J$, there is a chamber $h c_0^- = gx c_0^- \in \mc{R}_J$ such that $\delta(h c_0^-) = w w'$, which yields the other inclusion. This shows \textit{(\ref{lemma 1 item a})}. Now \textit{(\ref{lemma 1 item b})} follows from \textit{(\ref{lemma 1 item a})} and the fact that $w_J$ is the unique element of minimal length in $w \mc{W}_J$. 
\end{proof}

\begin{proposition} 
\label{prop filtration}
\label{is residual}
The sequence $(C_n)_{n \in \N}$ is a residual filtration of $\mc{C}^-(G^{\min}_R)$. 
\end{proposition}
\begin{proof}
Let $\mc{R}_J$ be a $J$-residue in $\mc{C}^-(G^{\min}_R)$. It is clear that for any $n \in \N$,  $\mc{R}_J \cap C_{n} \subseteq \mc{R}_J \cap C_{n+1}$. Moreover, $\bigcup_{n \in \N} (\mc{R}_J \cap C_n) = \mc{R}_J$.

All that remains to show is that for every $n > 0$ such that $C_{n-1} \cap \mc{R}_J \neq \emptyset$, there exists a $j \in J$ such that every chamber $c \in C_n \cap \mc{R}_J$ is $j$-adjacent to some chamber $c' \in C_{n-1} \cap \mc{R}_J$. There are two possible cases: either $C_{n-1} \cap \mc{R}_J = C_n \cap \mc{R}_J$, or there exists an element $gc_0^- \in (C_n \setminus C_{n-1})\cap\mc{R}_J$. In the former case, we may choose $j \in J$ arbitrarily, since every chamber $c \in C_n \cap \mc{R}_J = C_{n-1} \cap \mc{R}_J$ is $j$-adjacent to itself. In the latter case, by definition there is a fixed $w \in \mc{W}$ such that $\delta(g c_0^-) = w$, for all $g c_0^- \in (C_n\setminus C_{n-1})\cap \mc{R}_J$. Moreover, $w \mc{W}_J$ has a unique element of minimal length $w_J$, and since $C_{n-1} \cap \mc{R}_J$ is non-empty, this element is different from $w$. Indeed, define $m = \min \{r \in \N \mid C_r \cap \mc{R}_J \neq \emptyset\}$. Then $m < n$, and with $l$ as in \Cref{lemma gate properties}, i.e.\ $l = \ell(w_J)$, 
\begin{align*}
C_m \cap\mc{R}_J & = \left\{g c_0^- \in \mc{R}_J \mid \abs{gc_0^- } = m \right\} = \left\{gc_0^- \in \mc{R}_J \mid \ell(\delta(gc_0^-)) = l\right\} \\
& =  \left\{g c_0^- \in \mc{R}_J \mid \delta(gc_0^-) = w_J \right\}.
\end{align*}
In particular, by \Cref{lemma gate properties}, \[w = w_J w', \text{ with } w' \in \mc{W}_J \setminus \{1_{\mc{W}}\}.\]
Let now $j \in J$ such that $\ell(w's_j) = \ell(w')-1$, and let moreover $gc_0^- \in (C_n \setminus C_{n-1})\cap \mc{R}_J$, i.e.\ $\delta(gc_0^- ) = w$. Choose $y \in \varphi_{jR}(\SL_2(R) - U_2^+(L) B^-_2(R))$. We show that $\delta(gyc_0^-) = w s_j$. Indeed, by \Cref{lemma preimage B+2}, we have that $p(y) \not\in \varphi_{jk}(B^-_2(k))$. In particular, $p(y) \in \varphi_{ik}(B^-_2(k) \begin{psmallmatrix} 0 & 1 \\ -1 & 0\end{psmallmatrix} B^-_2(k))$, and hence $p(y) \in B^-_k \tilde{s}_j B^-_k$. It follows that 
\[p(gy) \in B^+_k \tilde{w} B^-_k \tilde{s}_j B^-_k = B^+_k \tilde{w} \tilde{s}_j B^-_k, \]
where the latter equality follows by \eqref{product Birkhoff classes}, since \[\ell(ws_j) = \ell(w_J w' s_j) = \ell(w_J) + \ell(w's_j) = \ell(w_J) + \ell(w') -1 = \ell(w) -1.\]
Hence $gy \in B(ws_j)$ by \eqref{inverse image Birkhoff classes}, i.e.\ $\delta(gy c_0^-) = w s_j$. In particular, \[\ell(\delta(gy c_0^-)) < \ell(\delta(gc_0^-))\]
and by definition of $\abs{\cdot} \colon \mc{W} \to \N$, it follows that $gy c_0^- \in C_{n-1}$. By construction $gyc_0^- \in \mc{R}_J$ is $j$-adjacent to $gc_0^-$. Hence $j \in J$ is an element such that every $gc_0^-\in C_n \cap \mc{R}_J$ is $j$-adjacent to some $hc_0^- \in C_{n-1} \cap \mc{R}_J$, finishing the proof.
\end{proof}

Next, we show the condition \eqref{lco}. We calculate what the set $\aff(\mc{R}_J)$ looks like for a residue $\mc{R}_J$ in $\mc{C}^-(G^{\min}_R)$. 

\begin{proposition}
\label{aff of residue}
Let $\mc{R}_J$ be a $J$-residue in $\mc{C}^-(G^{\min}_R)$, with $J \subseteq I$. Let $hc_0^-\in \mc{R}_J$ be such that $\delta(hc_0^-) = w$, and let $w_J$ be the unique element of minimal length in $w\mc{W}_J$. Then there is some $gc_0^-\in \mc{R}_J$ such that $\delta(gc_0^-) = w_J$, and
\begin{equation}
\aff(\mc{R}_J) = \left\{gx c_0^- \in \mc{C}^-(G^{\min}_R) \mid x \in \ov{U}^+_{JR}\right\}.
\end{equation} 
\end{proposition}

\begin{proof}
Note that a chamber $g c_0^- \in \mc{R}_J$ with $\delta(g c_0^- ) = w_J$ exists by \Cref{lemma gate properties}\eqref{lemma 1 item a}. In particular, $g c_0^- \in b_+ \ov{G}_{JR} c_0^-$ for some $b_+ \in \ov{B}^+_R$. Up to replacing $g c_0^-$ by $b_+ \tilde{w}_J c_0^-$ (which belongs to $\mc{R}_J(gc_0^-) = \mc{R}_J$ and satisfies $\delta(b_+ \tilde{w}_J c_0^-) = w_J$), we may assume that $g c_0^- = b_+ \tilde{w}_J c_0^-$. 
We then have 
\[\mc{R}_J = \left\{ b_+ \tilde{w}_J x c_0^- \in \mc{C}^-(G^{\min}_R) \mid x \in \ov{G}_{JR}\right\}.\]
Now, by \Cref{lemma gate properties}\eqref{lemma 1 item b}, \[b_+ \tilde{w}_J x  c_0^- \in \aff(\mc{R}_J) \iff p(b_+ \tilde{w}_J x) \in B^+_k \tilde{w}_J B^-_k. \]
Moreover, $G_{Jk}$ has a Birkhoff decomposition $G_{Jk} = \coprod_{w' \in \mc{W}_J} U^+_{Jk} \tilde{w}' B^-_{Jk}$. In particular, for any $x \in \ov{G}_{JR}$, there is a $w' \in \mc{W}_J$ such that $p(x) \in U^+_{Jk} \tilde{w}' B^-_k$, and hence
\[p(b_+ \tilde{w}_J x ) \in B^+_k \tilde{w}_J U^+_{Jk} \tilde{w}' B^-_{Jk}.\]
By \Cref{conjugation of UJ+ by J reduced}, since $w_J$ is right $J$-reduced, we have that $\tilde{w}_J U^+_{Jk} \subseteq B^+_k \tilde{w}_J$. 
In particular, 
\[ p(b_+ \tilde{w}_J x ) \in  B^+_k \tilde{w}_J U^+_{Jk} \tilde{w}' B^-_{k} \subseteq B_k^+ \tilde{w}_J \tilde{w}' B^-_k.\]
Hence $b_+\tilde{w}_J x c_0^- \in \aff(\mc{R}_J)$ if and only if $w' = 1$, i.e.\ $p(x) \in U^+_{Jk} B^-_{Jk}$. By \Cref{Kernel p}, it follows that $b_+ \tilde{w}_J x c_0^- \in \aff(\mc{R}_J)$ if and only if $x \in \ov{U}^+_{JR} \ov{U}^+_L \ov{B}^-_R \cap \ov{G}_{JR}$. Now,
\begin{align*}
\ov{U}^+_{JR} \ov{U}^+_L \ov{B}^-_R \cap \ov{G}_{JR}  = \ov{U}^+_{JR}\left( \ov{U}^+_L \ov{B}^-_R \cap \ov{G}_{JR}\right) & \subseteq \ov{U}^+_{JR} \left(\ov{B}^+_R \ov{B}^-_R \cap \ov{G}_{JR}\right) \\
& \subseteq \ov{U}^+_{JR} \left(\ov{U}^+_{JR} \ov{U}^-_{JR} T_R\right) \subseteq \ov{U}^+_{JR} \ov{B}^-_R,
\end{align*}
where the inclusion $\ov{B}^+_R \ov{B}^-_R \cap\ov{G}_{JR} \subseteq \ov{U}^+_{JR}\ov{U}^-_{JR} T_R$ follows from \Cref{5.10}. It follows in particular that $b_+ \tilde{w}_J x c_0^- \in \aff(\mc{R}_J)$ if and only if $x \in \ov{U}^+_{JR} \ov{B}^-_R$, showing the claim.
\end{proof}

\begin{proposition}
\label{satisfies lco}
Let $A$ be $2$-spherical, and $R$ a local ring satisfying \eqref{co}. Then the filtration $(C_n)_{n\geq 0}$ of $\mc{C}^-(G^{\min}_R)$ as defined in \eqref{definition filtration} satisfies \eqref{lco}.
\end{proposition}
\begin{proof}
By \Cref{aff of residue}, for any $J$-residue $\mc{R}_J$, there is some $g \in G^{\min}_R$ such that any two elements in $\aff(\mc{R}_J)$ are of the form $gx c_0^-$ and $gy c_0^-$, with $x,y \in \ov{U}^+_{JR}$. Since $R$ is assumed to satisfy \eqref{co}, it follows by \Cref{generated by simples 2 spherical}  that $x^{-1}y$ can be expressed as a product of elements in the simple root subgroups $U_{\alpha_j R}$ for $j \in J$. Hence there exists a $J$-gallery from $gx c_0^-$ to $gy c_0^-$, which is moreover completely contained in $\aff(\mc{R}_J)$ by \Cref{aff of residue}.
\end{proof}
We now prove \Cref{main theorem}.
\begin{theorem}
\label{main thm in text}
Let $A = (A_{ij})_{i,j \in I}$ be a $2$-spherical GCM and $R$ a local Bézout domain satisfying \eqref{co}. Then $\ov{U}^+_R$ is the amalgamated product of its subgroups $\ov{U}^+_{JR}$ for $J\subseteq I$ and $\abs{J} \leq 2$, along the canonical inclusions $\ov{U}^+_{J_1 R} \into \ov{U}^+_{J_2 R}$ for all $J_1 \subseteq J_2$, if this is the case for all of its rank $3$ subgroups $\ov{U}_{I' R}^+$, i.e.\ for all $I' \subseteq I$ with $\abs{I'} = 3$. 
\end{theorem}
\begin{proof}
Under the hypotheses on $A$ and $R$, the chamber system $\mc{C}^-(G^{\min}_R)$ is a building by \Cref{coincides w building of K}. In particular, $\mc{C}^-(G^{\min}_R)$ is simply connected by \Cref{buildings are sc}. The filtration $(C_n)_{n \in \N}$ of $\mc{C}^-(G^{\min}_R)$ as defined in \eqref{definition filtration}, is a residual filtration satisfying \eqref{lco} by \Cref{is residual} and \Cref{satisfies lco}. Moreover, by definition, $C_0 = \ov{U}^+_R c_0^-$. 
Hence if $(C_n)_{n \in \N}$ were to satisfy \eqref{lsco}, then by \Cref{314 devmu}, $\ov{U}^+_R c_0^-$ would be a simply connected chamber system, on which $\ov{U}^+_R$ acts transitively.

In particular, if $(C_n)_{n \in \N}$ satisfies \eqref{lsco}, we find that $\ov{U}^+_R$ has the desired amalgamated product decomposition by \Cref{Tits presentation}, since the stabilisers in $\ov{U}^+_R$ of the residues of rank at most $2$ containing $c_0^-$ are the subgroups $\ov{U}^+_R \cap \ov{P}^-_{JR} = \ov{U}^+_{JR}$ (see \eqref{intersection B+ P-}) with $\abs{J}\leq 2$, by \Cref{example chamber system structure}. 

By definition, $(C_n)_{n\in \N}$ satisfies \eqref{lsco} if and only if for any $J$-residue $\mc{R}_J \subseteq \mc{C}^-(G^{\min}_R)$ with $\abs{J} = 3$, the set $\aff(\mc{R}_J)$ is simply connected. By \Cref{aff of residue}, this is equivalent to the set $g \ov{U}^+_{JR} c_0^-$
being simply connected, for all $g \in G^{\min}_R$ and all $J \subseteq I$ of size $3$. But since the action of $G^{\min}_R$ on $\mc{C}^-(G^{\min}_R)$ preserves $i$-adjacency for any $i \in I$, we have that for any $g \in G^{\min}_R$ and any $J\subseteq I$, 
\[ g \ov{U}^+_{JR} c_0^- \cong \ov{U}^+_{JR} c_0^-, \]
as chamber systems. In combination with \Cref{Tits presentation}, it follows that $(C_{n})_{n \in \N}$ satisfies \eqref{lsco} if and only if for any $J \subseteq I$ with $\abs{J} = 3$, the subgroup $\ov{U}^+_{J R} \leqslant \ov{U}^+_R$ is the amalgamated product of its subgroups $\ov{U}^+_{J R} \cap \ov{P}^-_{J'R} = \ov{U}_{J' R}^+$ (see \eqref{intersection B+ P-}) with $J' \subseteq J$ and $\abs{J'}\leq 2$, along the inclusions $\ov{U}_{J_1 R}^+ \subseteq \ov{U}_{J_2 R}^+$ for all $J_1 \subseteq J_2 \subseteq J$.
\end{proof}

\section{\texorpdfstring{Simple connectedness in type $A_3$}{Simple connectedness in type A3}}
\label{sect A3}
\subsection{\texorpdfstring{The reduction step and proof in the $A_3$ case}{The reduction step and proof in the A3 case}}
In \cite{Tits86}, Tits showed that for a Cartan matrix $A$ (i.e.\ a GCM of \emph{finite} type), the problem of obtaining a presentation of $U^+_R$ as an amalgamated product of its rank $2$ subgroups, reduces to its rank $3$ subgroups having such a presentation. This reduction step holds over all rings satisfying \eqref{co} (compare with \Cref{main thm in text}).

He then moreover showed that for a Cartan matrix $A$ of type $A_3$, and $\K$ a field different from $\F_2$, the group $U^+_{\K}$ is the amalgam of its rank $2$ subgroups. These results were then combined to obtain an amalgamated product decomposition of $U^+_{\K}$ in terms of its rank $2$ subgroups for any Cartan matrix of simply laced type.

We will adapt Tits' strategy to the more general setting of commutative rings satisfying a minor condition which extends the requirement that $\K \neq \F_2$, proving \Cref{spherical main}.

First, we formulate the statement and proof of Tits' reduction step to the rank $3$ case for spherical type GCM's, for the convenience of the reader. For this, we recall the following fact.

\begin{lemma}[{\cite[Ch.VI, §1 Théorème 2 \textit{\&} Proposition 24]{Bourbaki}}]
\label{bases can be extended}
Let $\Phi$ be a finite crystallographic root system with simple roots $\{\alpha_1, \ldots,\alpha_n\}$, in a finite-dimensional Euclidean space $V$. Let $V'$ be a subspace of $V$ and $\Phi' := \Phi \cap V'$. Let $B'$ be a basis of the root system $\Phi'$. Then there exists a basis $B$ of $\Phi$ such that $B' \subseteq B$. In particular, there exists an element $w \in \mc{W}$ in the Weyl group of $\Phi$, such that $w \cdot B' \subseteq \{\alpha_1 , \ldots , \alpha_n\}$.
\end{lemma}

\begin{proposition}[{\cite[Proposition 7]{Tits86}}]
\label{reduction step Tits}
Let $A = (A_{ij})_{i,j \in I}$ be a Cartan matrix, with simple roots $\{\alpha_i \mid i \in I \}$. Let $R$ be a (commutative, unital) ring. Suppose the following two conditions are satisfied for $U^+_R \leqslant \mfG_A(R)$:
\begin{enumerate}[(i)]
\item For any distinct $i,j \in I$, the group $U^+_{\{i,j\} R}$ is generated by $U_{\alpha_i R}$ and $U_{\alpha_j R}$; \label{2 generated by simples}
\item For any distinct $i,j,l \in I$, the group $U^+_{\{i,j,l\} R}$ is the amalgamated product of its subgroups $U^+_{JR}$ with $J \subseteq \{i,j,l\}$ and $\abs{J} \leq 2$, along the canonical inclusions of the subgroups $U_{\alpha_i R}$, $U_{\alpha_j R}$ and $U_{\alpha_l R}$. \label{rank 3 amalgam Tits}
\end{enumerate}
Then $U^+_R$ itself is the amalgamated product of its subgroups $U^+_{J R}$ for all $J\subseteq I$ with $\abs{J} \leq 2$, along the canonical inclusions of the root subgroups $U_{\alpha_i R} \into U_{JR}^+$ for all $i \in J \subseteq I$.
\end{proposition}

\begin{proof}
Let $n = \abs{I}$. Denote by $\Phi = \Phi^+ \sqcup \Phi^-$ the (finite) root system of $A$, with Weyl group $\mc{W}$. Fix a total ordering $<$ on $\Phi^+$ such that $\htt(\beta) < \htt(\gamma)$ implies that $\beta < \gamma$. Let $N$ denote the size of $\Phi^+$, and write $\beta_1 < \beta_2 < \ldots < \beta_N$ for the roots in $\Phi^+$. We write moreover $\wt{U}$ for the amalgamated product of the groups $U_{\{i,j\}R}^+$, for all $i,j \in I$, along the inclusions $U_{\alpha_i R} \into U^+_{\{i,j\}R}$. We show the following property for all $1 \leq m \leq N$, by induction on $m$:
\begin{enumerate}
\item[(lift$_m$)] \hypertarget{lift}{There exists a unique system of “liftings”, i.e.\ injective group homomorphisms $\psi_i \colon U_{\beta_i} \to \wt{U}$ for all $i \leq m$, such that if $\beta_i = \alpha_j$ is a simple root, then $\psi_i$ is the canonical injection of $U_{\alpha_j}$ into $\wt{U}$, and such that if $\beta_i$, $\beta_j$ and $\beta_k$ are linearly independent for $1 \leq i,j,k \leq m$}, then $\psi_i$, $\psi_j$ and $\psi_k$ extend to a group homomorphism $\psi_{i,j,k} \colon\langle U_{\beta_i}, U_{\beta_j}, U_{\beta_k}\rangle \to \wt{U}$. 
\end{enumerate}
For $1 \leq m \leq n$, each $\beta_m$ is a simple root, so that $\psi_m$ is the canonical inclusion of $U_{\beta_m}$ in $\wt{U}$. The extensions $\psi_{i,j,k}$ with $1 \leq i,j ,k \leq n$ exist by \textit{(\ref{rank 3 amalgam Tits})}.  

Suppose now $m > n$. Then we can choose $i,j < m$ such that $\beta_m$ is a linear combination of $\beta_i$ and $\beta_j$ with positive integer coefficients, and such that $\{\beta_i,\beta_j\}$ is a basis of the root system $(\R \beta_i + \R\beta_j)\cap \Phi$. This follows by definition of the order $<$ on $\Phi^+$.

By \Cref{bases can be extended}, there exists a $w \in \mc{W}$ such that \[w\cdot \{\beta_i, \beta_j \} = \{\alpha_p, \alpha_q\} \subseteq \left\{\alpha_1, \ldots, \alpha_n\right\}. \]
Hence, by \textit{(\ref{2 generated by simples})},
\[\tilde{w} U_{\beta_m} \tilde{w}^{-1} \leqslant U^+_{\{p,q\} R} = \langle U_{\alpha_p}, U_{\alpha_q}\rangle = \tilde{w} \langle U_{\beta_i} , U_{\beta_j}\rangle\tilde{w}^{-1}.\] 
In particular, it follows that $U_{\beta_m} \leqslant \langle U_{\beta_i}, U_{\beta_j}\rangle$.

Let then $\psi_m$ be the restriction to $U_{\beta_m}$ of $\psi_{i,j} \colon \langle U_{\beta_i}, U_{\beta_j} \rangle \to \wt{U}$, where the latter homomorphism exists by the induction hypothesis. Then $\psi_m$ does not depend on the choice of $\beta_i$ and $\beta_j$. Indeed, suppose that $\beta_{i'}$ and $\beta_{j'}$ is a different pair of roots such that $i',j' <m$, and $\beta_m \in \Z \beta_{i'} + \Z \beta_{j'}$ and such that $\{\beta_{i'}, \beta_{j'}\}$ forms a basis of $(\R \beta_{i'} + \R \beta_{j'})\cap \Phi$. 

Then the root system  \[\Phi \cap (\R \beta_i + \R \beta_j + \R\beta_{i'} + \R\beta_{j'})\] has rank $3$, and we can choose a basis $\{\beta_l, \beta_p, \beta_q\}$ of this root system. By definition of the order $<$, we have that $\beta_l, \beta_p, \beta_q < \beta_m$. Let $H := \langle U_{\beta_l}, U_{\beta_p}, U_{\beta_q} \rangle \leqslant U^+_R$. By the induction hypothesis, there exists $\psi_{l,p,q}\colon H \to \wt{U}$ which extends the maps $\psi_{l}$, $\psi_p$ and $\psi_q$. By \Cref{bases can be extended} there exists some $w'\in \mc{W}$ such that \[w'\cdot \{\beta_l, \beta_p, \beta_q\} \subseteq \{\alpha_1, \ldots, \alpha_n\},\] and in particular we have by \textit{(\ref{2 generated by simples})} and \textit{(\ref{rank 3 amalgam Tits})} that $H$ contains $U_{\beta_i}$, $U_{\beta_{j}}$, $U_{\beta_{i'}}$ and  $U_{\beta_{j'}}$. Then $\psi_{l,p,q}$  restricts to $\psi_i$, $\psi_j$, $\psi_{i'}$ and $\psi_{j'}$ on respectively each of these. In particular, $\psi_m$ must be the restriction of $\psi_{l,p,q}$ to $U_{\beta_m}$, so that $\psi_m$ is independent of the choice of $\beta_i$ and $\beta_j$.


Let now $\beta_s,\beta_t< \beta_m$ be such that $\{\beta_s,\beta_t,\beta_m\}$ is linearly independent. For a basis $\{\beta_{i_1}, \beta_{i_2}, \beta_{i_3}\}$ of the root system generated by $\{\beta_s, \beta_t, \beta_m\}$, \Cref{bases can be extended} implies that there is some $w \in \mc{W}$ and some $1 \leq j_1, j_2,j_3 \leq n$ such that \[\tilde{w} \langle U_{\beta_{i_1}} , U_{\beta_{i_2}}, U_{\beta_{i_3}} \rangle \tilde{w}^{-1} \leqslant \langle U_{\alpha_{j_1}}, U_{\alpha_{j_2}}, U_{\alpha_{j_3}}\rangle.\]
In particular, by \textit{(\ref{2 generated by simples})} and \textit{(\ref{rank 3 amalgam Tits})}, 
\[\langle U_{\beta_s}, U_{\beta_t}, U_{\beta_m} \rangle \leqslant \langle U_{\beta_{i_1}} , U_{\beta_{i_2}}, U_{\beta_{i_3}} \rangle \leqslant \tilde{w}^{-1} \langle U_{\alpha_{j_1}}, U_{\alpha_{j_2}}, U_{\alpha_{j_3}}\rangle \tilde{w}. \]
Since the $\{\alpha_{j_1}, \alpha_{j_2},\alpha_{j_3}\}$ are simple roots, we moreover have a map $\psi_{j_1,j_2,j_3 } \colon \langle U_{\alpha_{j_1}},U_{\alpha_{j_2}},U_{\alpha_{j_3}} \rangle \to \wt{U}$ extending the inclusions $U_{\alpha_{\ell}} \into \wt{U}$ for $\ell \in \{j_1, j_2,j_3\}$. Then the map 
\[ \psi_{s,t,m} \colon \langle U_{\beta_s},U_{\beta_t},U_{\beta_m}\rangle  \to \wt{U}, \quad u \mapsto \psi_{j_1,j_2,j_3}(\tilde{w}u \tilde{w}^{-1})\]
restricts to $\psi_{s}$, $\psi_{t}$ and $\psi_{m}$ on respectively $U_{\beta_s}$, $U_{\beta_t}$ and $U_{\beta_m}$. This shows (\hyperlink{lift}{lift$_{m}$}).

Since $\wt{U}$ is generated by its subgroups $U_{\alpha_i}$ for $i \in I$, there is a canonical group homomorphism $\Psi \colon \wt{U} \to U^+_R$ induced by mapping each $U_{\alpha_i}$ isomorphically onto its image in $U^+_R$. By \Cref{generated by simples 2 spherical}, $\Psi$ is surjective.

Moreover, we claim that the image of any $\psi_k(x_{\beta_k}(r_k)))$ under $\Psi$ is given by $x_{\beta_k}(r_k)$. We show this by induction on $\htt(\beta_k)$. If $\htt(\beta_k) = 1$, the claim holds by definition. Suppose $\htt(\beta_k) = n>1$. Write $\beta_k$ as a linear combination of distinct $\beta_i$ and $\beta_j$ with $\htt(\beta_i), \htt(\beta_j) < \htt(\beta_k)$. We have again by \textit{(\ref{2 generated by simples})} that $U_{\beta_k} \leqslant \langle U_{\beta_i}, U_{\beta_j}\rangle$. Let $\psi$ be the extension of $\psi_{i}$ and $\psi_j$ to $\langle U_{\beta_i},U_{\beta_j}\rangle$ given by (\hyperlink{lift}{lift$_{N}$}). As before, the restriction of $\psi$ to $U_{\beta_k}$ coincides with $\psi_k$. Then for $x \in U_{\beta_k}$, we have that $\Psi(\psi_k(x)) = \Psi(\psi(x)) = x$, since $x$ is a product of elements of the form $x_{\beta_i}(r_i)$ and $x_{\beta_j}(r_j)$, for which the claim holds by the induction hypothesis.


By (\hyperlink{lift}{lift$_{N}$}) and \eqref{commutation KM min}, for any pair of distinct roots $\beta_n,\beta_m \in \Phi^+$, the following relations hold in $\wt{U}$:
\[ [\psi_n(x_{\beta_n}(r)),\psi_m(x_{\beta_m}(s))] = \prod_{\gamma} \psi_{\gamma}(x_{\gamma}(C_{ij}^{\beta_n \beta_m} r^{i} s^{j})), \quad \text{ for all } r,s \in R,  \]
where $\gamma = i \beta_n + j \beta_m$ runs through $]\beta_n,\beta_m[_{\N}$ in the order $<$. In particular, any element $u \in \wt{U}$ has an expression 
\[ u = \prod_{i = 1}^N \psi_i(x_{\beta_i}(r_i)), \quad \text{ with } r_i \in R.  \]
By the previous paragraph, the image of $u$ by $\Psi$ is given by $\prod_{i = 1}^{N} x_{\beta_i}(r_i)$. \Cref{unique expression U+} now implies that the above expression of $u$ is unique and that the canonical map $\Psi \colon \wt{U} \to U^+_R$ is injective. 
\end{proof}

\begin{remark}
\label{rk generated by 2 simples}
The condition \textit{(\ref{2 generated by simples})} holds as soon as $R$ satisfies \eqref{co}, as stated in \Cref{generated by simples 2 spherical} (this particular implication being a direct consequence of \cite[Lemma 11.1]{Allcock}). In particular, for a simply laced Cartan matrix, the condition  \textit{(\ref{2 generated by simples})} is satisfied for any commutative unital ring $R$.
\end{remark}

For the rest of this section, let $A$ be of type $A_3$. Note that in this case, for any $R$,
\[G_R = \mathrm{E}_4(R) \leqslant \SL_4(R), \quad \text{ and } \quad U^+_{R} = \left\{\begin{psmallmatrix} 1 & \ast & \ast & \ast \\ 0 & 1 & \ast & \ast \\ 0 & 0 & 1 & \ast \\ 0 & 0 & 0 & 1 \end{psmallmatrix} \in \mathrm{E}_4(R) \right\}.\]
Here $\mathrm{E}_4(R)$ denotes the subgroup of $\SL_4(R)$ generated by the elementary matrices $I_4 + \lambda E_{i,j}$, for all $\lambda \in R$ and $i \neq j$, where $I_4$ denotes the identity matrix and $E_{i,j}$ denotes the matrix which has every entry equal to $0$, except the entry $(i,j)$, which equals $1$. For any $i \in I = \{1,2,3\}$, the root subgroup $U_{\alpha_i}$ of $U^+_R$ consists of the elements $I_4 + \lambda E_{i,i+1}$, for all $\lambda \in R$.

\begin{theorem}
\label{case A3}
Let $R$ be a (commutative, unital) ring such that $1_R$ is the sum of two invertible elements. Let $A = (A_{ij})_{i,j \in I}$ be the Cartan matrix of type $A_3$. Then $U^+:= U^+_R \leqslant G_R$ is the amalgamated product of the simple root subgroups $U_{\alpha_i}$ and the rank two subgroups $U_{JR}^+ = \langle U_{\alpha_i}, U_{\alpha_j} \rangle$, along the canonical inclusions $U_{\alpha_i } \into U_{J R}^+$, for all $J = \{i,j\} \subseteq I$.
\end{theorem}

When $R$ is a field, this statement is precisely \cite[Proposition 8]{Tits86}, which we formulate here for completeness' sake.

\begin{proposition}[{\cite[Proposition 8]{Tits86}}]
\label{Tits case A3}
Let $\K$ be a field different from $\F_2$ and let $A = (A_{ij})_{i,j \in I}$ be the Cartan matrix of type $A_3$. Then $U^+_{\K} \leqslant G_{\K} =\SL_4(\K)$ is the amalgamated product of the simple root subgroups $U_{\alpha_i}$ and the rank two subgroups $U^+_{J\K} = \langle U_{\alpha_i}, U_{\alpha_j}\rangle$, along the canonical inclusions $U_{\alpha_i} \into U_{J \K}^+$, for all $J =\{i,j \} \subseteq I$.
\end{proposition}

Before we proceed with the proof of \Cref{case A3}, we first introduce some notation. By definition, $U^+ := U^+_R$ is the group with generators \[x_{\alpha}(r), \text{ for all } \alpha \in \Delta := \{\alpha_1, \alpha_2, \alpha_3, \alpha_{1} + \alpha_2, \alpha_2 + \alpha_3, \alpha_1 + \alpha_2 + \alpha_3\} \text{ and all } r \in R,\] and relations, for all $r,s \in R$:
\newcounter{savedoldequation}
\setcounter{savedoldequation}{\value{equation}}
\counterwithout{equation}{section}
\renewcommand{\theequation}{\roman{equation}}
\setcounter{equation}{0}
\begin{align}
& x_{\alpha}(r) x_{\alpha}(s)  = x_{\alpha}(r+s), \text{ for all } \alpha \in \Delta, \label{R+}\\
[& x_{\alpha_1}(r),x_{\alpha_2}(s)] = x_{\alpha_1 + \alpha_2}(rs), \quad [x_{\alpha_1 + \alpha_2}(r),x_{\alpha_1}(s)] = [x_{\alpha_1 + \alpha_2}(r),x_{\alpha_2}(s) ] = 1, \label{1+2} \\
[& x_{\alpha_2}(r),x_{\alpha_3}(s)]  = x_{\alpha_2 + \alpha_3}(rs), \quad [x_{\alpha_2 + \alpha_3}(r),x_{\alpha_2}(s)] = [x_{\alpha_2 + \alpha_3}(r), x_{\alpha_3}(s)] = 1, \label{2+3}\\
[ &  x_{\alpha_1}(r),x_{\alpha_3}(s)] = 1 \label{1 commuteert met 3}\\
[&  x_{\alpha_1+\alpha_2}(r), x_{\alpha_2+\alpha_3}(s)]  = 1 \label{12 commutes with 23} \\
[& x_{\alpha_1+ \alpha_2}(r), x_{\alpha_3}(s)] = [x_{\alpha_1}(r),x_{\alpha_2+\alpha_3}(s)] \label{(1+2)+3=1 +(2+3)}\\
[& x_{\alpha_1 + \alpha_2}(r), x_{\alpha_3}(s)] = x_{\alpha_1 + \alpha_2 + \alpha_3}(rs),\label{lands in 1+2+3}\\
[& x_{\alpha_1 + \alpha_2 + \alpha_3}(r), x_{\alpha}(s)] = 1, \text{ for all } \alpha \in \Delta.\label{123 central}
\end{align}
\setcounter{equation}{\value{savedoldequation}}
\renewcommand{\theequation}{\arabic{equation}}
\counterwithin{equation}{section}
Moreover, the amalgamated product $\wt{U}$ of the root subgoups $U_{\alpha_i}$ and the rank two subgroups $U_{i,j} := U_{\{i,j \}R}^+$ of $U^+$ along the inclusions of $U_{\alpha_i }$ into $U_{i,j}$, is given by the group with the generators $x_{\alpha}(r)$, for all $r \in R$ and all $\alpha \in \tilde{\Delta} := \Delta \setminus \{\alpha_1+\alpha_2+\alpha_3\}$, satisfying the relations \eqref{R+} for all $\alpha \in \tilde{\Delta}$, as well as the relations \eqref{1+2}, \eqref{2+3} and  \eqref{1 commuteert met 3}.

Define in $\wt{U}$, for any $t \in R$,
\begin{equation}
\label{v123}
v_{123}(t) := [x_{\alpha_1+\alpha_2}(1),x_{\alpha_3}(t)].
\end{equation}

Proving \Cref{case A3} consists of showing that $\wt{U}$ also satisfies \cref{12 commutes with 23,(1+2)+3=1 +(2+3),lands in 1+2+3,123 central}, with the elements $x_{\alpha_1+\alpha_2+\alpha_3}(t)$ replaced by $v_{123}(t)$, and that $v_{123}(t)$ also satisfies \eqref{R+}.

Let $U_i := U_{\alpha_i}$ for all $i \in I$. Moreover, let $V_i$ be the conjugation of $U_i$ by $x_{\alpha_2}(1)$, i.e.\  $V_i := x_{\alpha_2}(-1) U_i x_{\alpha_2}(1)$. In particular, by \eqref{1+2} and \eqref{2+3},
\begin{align}
V_1 &= \langle x_{\alpha_1}(r) x_{\alpha_1 + \alpha_2}(r) \mid r \in R \rangle \label{description V1}\\
V_3 & = \langle x_{\alpha_3}(r) x_{\alpha_2 + \alpha_3}(-r) \mid r \in R \rangle.\label{description V3}
\end{align}

\begin{remark}
\label{remark Tits}
As Tits remarked in \cite[§16.5]{Tits86}, \Cref{Tits case A3} does not hold over $\F_2$. Indeed, if one supposes $\K = \F_2$, then we define, for $i \in \{1,3\}$, the groups $D_i := U_{i} \ast V_i$, which are isomorphic to the infinite dihedral group $D_{\infty}$. Moreover, we can define the semidirect product
\[ H :=  (D_1 \times D_3)\rtimes U_{2}, \]
by letting the elements $x_{\alpha_2}(1)$ act on $D_i$ by conjugation (respecting the relations \eqref{1+2} and \eqref{2+3}). Then the relations \eqref{R+}, \eqref{1+2}, \eqref{2+3} and \eqref{1 commuteert met 3} are satisfied in $H$, and in particular the infinite group $H$ is a quotient of $\wt{U}$. It follows that $\wt{U} \not\cong U$, since $U$ is the finite group \[U \cong \left\{ \begin{psmallmatrix} 1 & \ast & \ast & \ast \\ 0 & 1 & \ast & \ast \\ 0 & 0 & 1 & \ast \\ 0 & 0 & 0 & 1 \end{psmallmatrix} \in \SL_4(\F_2)\right\}.\]
\end{remark}

The following two lemmas are preliminary steps towards the proof of \Cref{case A3}.
\begin{lemma}[{\cite[§16 (7)]{Tits86}}]
\label{commutation UV}
For $U_i$ and $V_i$ as above, the following commutation relations hold in~$\wt{U}$:
\begin{equation*}
[U_1, U_3] = [U_3, V_3] = [V_1,V_3] = [U_1,V_1] = \{1\}.
\end{equation*}
\end{lemma}
\begin{proof}
The groups $U_1$ and $U_3$ commute by \eqref{1 commuteert met 3}, and conjugating relation \eqref{1 commuteert met 3} by $x_{\alpha_2}(1)$ shows that $V_1$ and $V_3$ commute. Now $U_1$ commutes with $V_1$ by \eqref{description V1} and \eqref{1+2}, and $U_3$ commutes with $V_3$ by \eqref{description V3} and \eqref{2+3}.
\end{proof}
Moreover, the following holds.
\begin{lemma}[{\cite[Lemme 4]{Tits86}}]
\label{Tits lemma central elt X}
Let $u_i\in U_i$ and $v_i \in V_i$ be such that $[u_1 v_1, u_3 v_3]~=~1$. Then \[ [u_1, v_3] = [u_3, v_1], \]
and $[u_1, v_3]$ is central in the subgroup $X$ of $\wt{U}$ generated by $U_1$, $U_3$, $V_1$ and $V_3$.
\end{lemma}
\begin{proof}
Suppose $u_1 v_1 u_3 v_3 = u_3 v_3 u_1 v_1$. Then $v_1 u_3 v_3 = u_1^{-1} u_3 v_3 u_1 v_1$. But $v_1 u_3 v_3 = v_3 v_1 u_3$ since $v_3$ commutes with $u_3$ and $v_3$ by \Cref{commutation UV}. Hence \[v_1 u_3 = v_3^{-1} u_1^{-1} u_3 v_3 u_1 v_1.\] 
But the latter is equal to $u_3 v_3^{-1} u_1^{-1} v_3 u_1 v_1$, since $u_3$ commutes with $u_1$ and $v_3$. Hence \[u_3^{-1} v_1 u_3 = v_3^{-1} u_1^{-1} v_3 u_1 v_1 = v_1 v_3^{-1} u_1^{-1} v_3 v_1,\]
where the last equality follows from the fact that $v_1$ commutes with $u_1$ and $v_3$. Hence \[v_1^{-1} u_3^{-1} v_1 u_3 = v_3^{-1} u_1^{-1} v_3 u_1.\]
This shows the first statement. It now trivially follows that $[u_1,v_3]$ is central in $X$, since, by \Cref{commutation UV}, $[u_1,v_3]$ commutes with $U_3$ and $V_1$, and $[u_3,v_1]$ commutes with $U_1$ and $V_3$.
\end{proof}

\begin{proof}[Proof of \Cref{case A3}]
Consider the group $\wt{U}$ as defined above. We show the relations \cref{12 commutes with 23,(1+2)+3=1 +(2+3),lands in 1+2+3,123 central} hold, with the elements $x_{\alpha_1+\alpha_2+\alpha_3}(t)$ replaced by $v_{123}(t)$, and that $v_{123}(t)$ also satisfies \eqref{R+}. We begin by showing relation \eqref{(1+2)+3=1 +(2+3)}. Let $r,s \in R$. Then by \eqref{1+2}, \[[x_{\alpha_1}(r),x_{\alpha_3}(s)] = 1.\] Conjugating this expression by $x_{\alpha_2}(t)$ for $t \in R$ implies that
\[[x_{\alpha_1}(r)x_{\alpha_1+\alpha_2}(rt),x_{\alpha_3}(s)x_{\alpha_2+\alpha_3}(-st)] = 1.\]
In particular, by \eqref{R+},
\[ [x_{\alpha_1}(r-rt) x_{\alpha_1}(rt) x_{\alpha_1+\alpha_2}(rt), x_{\alpha_3}(s -st) x_{\alpha_3}(st)x_{\alpha_2+\alpha_3}(-st)] = 1. \]
But by \eqref{description V1} and \eqref{description V3}, 
\begin{align*}
x_{\alpha_1}(r-rt) \in U_1,  \quad & x_{\alpha_1}(rt) x_{\alpha_1+\alpha_2}(rt) \in V_1,\text{ and }\\ 
x_{\alpha_3}(s-st) \in U_3, \quad & x_{\alpha_3}(st) x_{\alpha_2+\alpha_3}(-st) \in V_3.
\end{align*}
Hence we may apply \Cref{Tits lemma central elt X} to obtain that for all $r,s,t \in R$
\begin{equation}
\label{star}
\left[x_{\alpha_1}(r-rt), x_{\alpha_3}(st) x_{\alpha_2+\alpha_3}(-st)\right] = \left[x_{\alpha_3}(s-st),x_{\alpha_1}(rt)x_{\alpha_1+\alpha_2}(rt)\right],
\end{equation}
and moreover these elements are central in the subgroup $X$ of $\wt{U}$ generated by $U_1$, $U_3$, $V_1$ and $V_3$. In any group $G$, the following relations hold for $x,y,z \in G$:
\begin{align}
[x,zy] & = [x,y][x,z]^{y} \label{product second entry}\\
[xz,y] & = [x,y]^z [z,y], \label{product first entry}
\end{align}
It follows that 
\begin{align*}
\left[x_{\alpha_1}(r-rt), x_{\alpha_3}(st) x_{\alpha_2+\alpha_3}(-st)\right] & = \left[x_{\alpha_1}(r-rt), x_{\alpha_{2}+\alpha_3}(-st)\right]\left[x_{\alpha_1}(r-rt),x_{\alpha_3}(st)\right]^{x_{\alpha_2+\alpha_3}(-st)}\\
& = \left[x_{\alpha_1}(r-rt), x_{\alpha_{2}+\alpha_3}(-st)\right],
\end{align*}
where the last equality follows from \eqref{1 commuteert met 3}. Similarly, 
\[
\left[x_{\alpha_3}(s-st),x_{\alpha_1}(rt)x_{\alpha_1+\alpha_2}(rt)\right] = \left[ x_{\alpha_3}(s-st), x_{\alpha_1+\alpha_2}(rt)\right]. 
\]
However, in any group $G$ there are moreover the relations
\begin{equation}
\label{switch entries of commutators modulo inverses}
[x^{-1},y] = [y,x]^{x^{-1}} \text{ and } [x,y^{-1}] = [y,x]^{y^{-1}}, \qquad \text{ for all } x,y,z \in G.
\end{equation}
Applying \eqref{switch entries of commutators modulo inverses} to the previous expression, we get
\begin{align*}
\left[ x_{\alpha_3}(s-st),x_{\alpha_1+\alpha_2}(rt)\right] & = \left[x_{\alpha_1 + \alpha_2}(rt),x_{\alpha_3}(-s+st)\right]^{x_{\alpha_1+\alpha_2}(s-st)} \\
& = \left[x_{\alpha_1 + \alpha_2}(rt),x_{\alpha_3}(-s+st)\right], 
\end{align*}
where the last equality holds because $\left[x_{\alpha_1 + \alpha_2}(rt),x_{\alpha_3}(-s+st)\right]$ is central in $X$, and because $x_{\alpha_1+\alpha_2}(s-~st) \in X$. We conclude from \eqref{star} that 
\begin{equation}
\label{almost iv}
\forall r,s,t \in R, \quad \left[ x_{\alpha_1}(r-rt), x_{\alpha_2+\alpha_3}(-st) \right] = \left[x_{\alpha_1+\alpha_2}(rt),x_{\alpha_3}(-s+st)\right] .
\end{equation}
In particular, choosing $t \in R^{\times}$ such that $1 - t \in R^{\times}$, we have
\begin{align}
\label{plenty central elements in X}
\text{ } & \text{For any } r,s \in R, \text{ the elements } [x_{\alpha_1}(r), x_{\alpha_2+\alpha_3}(s)]  \\ & \text{and } [x_{\alpha_1 + \alpha_2}(r), x_{\alpha_3}(s)] \text{ are central in } X. \notag 
\end{align}
We now apply a change of variables to \eqref{almost iv}. If $\lambda \in R^{\times}$ is such that $1+ \lambda$ is invertible, then $t := \lambda(1+\lambda)^{-1}$ and  $1-t = (\lambda+1)^{-1}$ are invertible. In particular, by \eqref{almost iv},
\begin{align}
\label{if there is a+1 invertible}
& \text{For any } \lambda \in R^{\times} \text{ such that } \lambda+1\in R^{\times}, \text{ and for any } r,s \in R, \\ & [x_{\alpha_1}(r),x_{\alpha_2+\alpha_3}(s)] = [x_{\alpha_1+\alpha_2}(r \lambda) , x_{\alpha_3}(s\lambda^{-1})].\notag
\end{align}
Fix now an element $b \in R^{\times}$ such that $1+b \in R^{\times}$ (which exists by assumption). Then $1 + (-b-1) \in R^{\times}$, and $1 - b^{-1}(b+1) = b^{-1} \in R^{\times}$. In particular by \eqref{if there is a+1 invertible}, one has that, for any $r,s \in R$,
\begin{align*}
[x_{\alpha_1+\alpha_2}(r), x_{\alpha_3}(s)] & = [x_{\alpha_1}(r b^{-1}), x_{\alpha_2+\alpha_3}(s b)] \\
& =  [x_{\alpha_1+\alpha_2}(-r b^{-1}(b+1)), x_{\alpha_3}(-s b(b+1)^{-1})] \\
& = [x_{\alpha_1}(r), x_{\alpha_2+\alpha_3}(s)],
\end{align*}
and relation \eqref{(1+2)+3=1 +(2+3)} follows.

Next, we show \eqref{12 commutes with 23}. By definition, for any $r, s\in R$, 
\begin{align*}
\left[x_{\alpha_1+\alpha_2}(r),x_{\alpha_2+\alpha_3}(s) \right] & = \left[ [x_{\alpha_1}(r),x_{\alpha_2}(1)], [x_{\alpha_3}(-s),x_{\alpha_2}(1)]   \right] \\
& = \left[x_{\alpha_1}(-r)  x_{\alpha_2}(-1)x_{\alpha_1}(r)x_{\alpha_2}(1) ,x_{\alpha_3}(s)  x_{\alpha_2}(-1) x_{\alpha_3}(-s) x_{\alpha_3}(1)\right].
\end{align*}
Denote 
\begin{align*}
u_1  & := x_{\alpha_1}(-r) \in U_1, \quad v_1 := x_{\alpha_2}(-1) x_{\alpha_1}(r) x_{\alpha_2}(1) \in V_1, \\
u_3  & := x_{\alpha_3}(s) \in U_3, \quad v_3  := x_{\alpha_2}(-1) x_{\alpha_3}(-s) x_{\alpha_2}(1) \in V_3.
\end{align*}
Then we obtain by \eqref{product first entry} that 
\begin{align*}
\left[x_{\alpha_1+\alpha_2}(r),x_{\alpha_2+\alpha_3}(s) \right] & = [u_1, u_3 v_3]^{v_1} [v_1, u_3 v_3] = [u_1,v_3]^{v_1} [u_1,u_3]^{v_3 v_1} [v_1,u_3v_3] \\
& = [u_1,v_3] [v_1,u_3v_3],
\end{align*}
where the last equality holds because $u_1$ commutes with $u_3$, and $v_1$ commutes with $u_1$ and $v_3$ by \Cref{commutation UV}. Applying \eqref{product second entry} we then obtain that 
\begin{align*}
\left[x_{\alpha_1+\alpha_2}(r),x_{\alpha_2+\alpha_3}(s) \right]  & = [u_1,v_3] [v_1,v_3][v_1,u_3]^{v_3} \\ 
& = [u_1,v_3][v_1,u_3],
\end{align*}
where the latter equality follows from \Cref{commutation UV}. In particular, by \eqref{1+2} and \eqref{2+3},
\begin{align*}
\left[x_{\alpha_1+\alpha_2}(r),x_{\alpha_2+\alpha_3}(s) \right] & = \left[x_{\alpha_1}(-r), x_{\alpha_2}(-1) x_{\alpha_3}(-s) x_{\alpha_2}(1)\right]\left[ x_{\alpha_2}(-1) x_{\alpha_1}(r) x_{\alpha_2}(1) , x_{\alpha_3}(s)\right] \\
& = \left[x_{\alpha_1}(-r), x_{\alpha_3}(-s) x_{\alpha_2+\alpha_3}(s)\right] \left[ x_{\alpha_1}(r) x_{\alpha_1+\alpha_2}(r), x_{\alpha_3}(s)\right].
\end{align*}
Applying once more \eqref{product second entry} and \eqref{product first entry}, as well as \Cref{commutation UV}, yields
\[ \left[x_{\alpha_1+\alpha_2}(r),x_{\alpha_2+\alpha_3}(s) \right] = \left[x_{\alpha_1}(-r),x_{\alpha_2+\alpha_3}(s)\right] \left[x_{\alpha_1+ \alpha_2}(r), x_{\alpha_3}(s)\right]. \]
By \eqref{(1+2)+3=1 +(2+3)}, the last expression is the same as 
\[\left[x_{\alpha_1+\alpha_2}(r),x_{\alpha_2+\alpha_3}(s) \right] = \left[x_{\alpha_1}(-r),x_{\alpha_2+\alpha_3}(s)\right] \left[x_{\alpha_1}(r), x_{\alpha_2+\alpha_3}(s)\right].\]
Because by assumption there is an element $t \in R$ with $t , 1-t \in R^{\times}$, it follows from \eqref{plenty central elements in X} that $[x_{\alpha_1}(r),x_{\alpha_2+\alpha_3}(s)]$ is central in $X$, and hence in particular commutes with $x_{\alpha_1}(-r)$. Hence by \eqref{switch entries of commutators modulo inverses}, 
\[\left[x_{\alpha_1+\alpha_2}(r),x_{\alpha_2+\alpha_3}(s) \right] = \left[x_{\alpha_1}(-r),x_{\alpha_2+\alpha_3}(s)\right] \left[x_{\alpha_2+\alpha_3}(s), x_{\alpha_1}(-r)\right] = 1,\]
where the last equality is satisfied because $[x,y]^{-1} = [y,x]$ for any $x,y$ in a group $G$. The relation \eqref{12 commutes with 23} follows.

We now show \eqref{123 central}. Let $v_{123}(r)$ be as in \eqref{v123}. It suffices to show that $v_{123}(r)$ commutes with $x_{\alpha_2}(s)$ for all $r,s \in R$, since $v_{123}(r) = [x_{\alpha_1+\alpha_2}(1),x_{\alpha_3}(r)]$ is central in $X$ by \eqref{plenty central elements in X}, and $X$ contains all the root subgroups associated to the roots $\{\alpha_1, \alpha_3, \alpha_1+\alpha_2,\alpha_2+\alpha_3, \alpha_1+\alpha_2+\alpha_3\}$. But
\begin{align*}
v_{123}(r)^{x_{\alpha_2}(s)} & = \left[x_{\alpha_1+\alpha_2}(1)^{x_{\alpha_2}(s)}, x_{\alpha_3}(r)^{x_{\alpha_2}(s)}\right] \\
& = [x_{\alpha_1+\alpha_2}(1), x_{\alpha_3}(r)x_{\alpha_2+\alpha_3}(-rs)] \\
& = [x_{\alpha_1+\alpha_2}(1), x_{\alpha_2+\alpha_3}(-rs)] [x_{\alpha_1+\alpha_2}(1), x_{\alpha_3}(r)]^{x_{\alpha_2+\alpha_3}(-rs)} \\
& = [x_{\alpha_1+\alpha_2}(1), x_{\alpha_3}(r)] = v_{123}(r).
\end{align*}
Here the second equality holds by \eqref{1+2} and \eqref{2+3}, the third by \eqref{product second entry} and the fourth by \eqref{12 commutes with 23} and \eqref{plenty central elements in X}.

We now show \eqref{lands in 1+2+3}, with $x_{\alpha_1+\alpha_2+\alpha_3}(rs)$ replaced by $v_{123}(rs) = [x_{\alpha_1+\alpha_2}(1), x_{\alpha_3}(rs)]$. Let $r,s \in R$. Then 
\begin{align*}
[x_{\alpha_1+\alpha_2}(r), x_{\alpha_3}(s)]^{x_{\alpha_2}(-r)} & = \left[ [x_{\alpha_1}(1), x_{\alpha_2}(r)],x_{\alpha_3}(s)\right]^{x_{\alpha_2}(-r)} \\
& = \left[ x_{\alpha_2}(-r), [x_{\alpha_3}(s), x_{\alpha_1}(-1)]\right]^{x_{\alpha_1}(1)} \cdot \left[ x_{\alpha_1}(1), [x_{\alpha_2}(-r), x_{\alpha_3}(-s)]\right]^{x_{\alpha_3}(s)} \\
& = \left[ x_{\alpha_1}(1), [x_{\alpha_2}(-r), x_{\alpha_3}(-s)]\right]^{x_{\alpha_3}(s)} \\
& = \left[ x_{\alpha_1}(1), [x_{\alpha_2}(-r), x_{\alpha_3}(-s)]\right] \\
& = \left[ x_{\alpha_1}(1), [x_{\alpha_2}(1), x_{\alpha_3}(rs)]\right] \\
& = [x_{\alpha_1}(1), x_{\alpha_2+\alpha_3}(rs)] = [x_{\alpha_1+\alpha_2}(1), x_{\alpha_3}(rs)] = v_{123}(rs).
\end{align*}
Here the first equality follows by \eqref{1+2}, the second by the Hall--Witt identity, the third by \eqref{1 commuteert met 3}, the fourth by \eqref{1 commuteert met 3} and \eqref{2+3}, the fifth by \eqref{2+3}, and the last line by \eqref{(1+2)+3=1 +(2+3)}. But by \eqref{123 central}, \[ [x_{\alpha_1+\alpha_2}(r), x_{\alpha_3}(s)] = v_{123}(rs)^{x_{\alpha_2}(r)} = v_{123}(rs),\]
and \eqref{lands in 1+2+3} follows. Finally, for any $r,s \in R$,
\begin{align*}
v_{123}(r+s) & = [x_{\alpha_1 + \alpha_2}(1), x_{\alpha_3}(r)x_{\alpha_3}(s)] \\
& = [x_{\alpha_1+\alpha_2}(1), x_{\alpha_3}(s)] [x_{\alpha_1+\alpha_2}(1),x_{\alpha_3}(r)]^{x_{\alpha_3}(s)} \\
& = [x_{\alpha_1+\alpha_2}(1), x_{\alpha_3}(s)] [x_{\alpha_1+\alpha_2}(1),x_{\alpha_3}(r)] = v_{123}(s)v_{123}(r),
\end{align*}
where the second equality follows from \eqref{product second entry}, and the third equality holds since $v_{123}(r)$ is central in $X$, which contains $x_{\alpha_3}(r)$. Hence \eqref{R+} holds for $v_{123}(r)$, and this concludes the proof.
\end{proof}

We are now able to prove \Cref{spherical main}.
\begin{corollary}
\label{presentation simply laced spherical}
Let $R$ be a (commutative, unital) ring such that $1_R$ is the sum of two invertible elements, and let $A = (A_{ij})_{i,j \in I}$ be spherical and simply laced. Then $U^+_R \leqslant G_R$ is the amalgamated product of its subgroups $U^+_{JR}$ for $J \subseteq I$ and $\abs{J} \leq 2$, along the canonical inclusions $U_{J_1 R}^+ \leqslant U_{J_2 R}^+$ for $J_1 \subseteq J_2 \subseteq I$.
\end{corollary}
\begin{proof}
Since $A$ is simply laced, it follows by \Cref{rk generated by 2 simples} that the groups $U_{J R}^+$ with $\abs{J} = 2$ are generated by the root subgroups $U_{\alpha_j}$ for $j \in  J$. Hence the statement follows immediately from \Cref{reduction step Tits} and \Cref{case A3}.
\end{proof}

\subsection{\texorpdfstring{On the condition that $1$ is the sum of two units}{On the condition that 1 is the sum of two units}}
A unital ring $R$ is said to be \emph{semi-local} if $R/J(R)$ is semisimple, with $J(R)$ its Jacobson radical. For commutative rings this is equivalent to having finitely many maximal ideals by the Chinese remainder theorem. The following result is classical, as it follows immediately from \cite{Zelinsky} and \cite{Goldsmith}. We include the proof for convenience of the reader.
\begin{proposition}
\label{exists element a+1 a2+1}
Let $R$ be a (not necessarily commutative) semi-local unital ring. Then the following are equivalent:
\begin{enumerate}[(i)]
\item $R$ has no quotient isomorphic to $\F_2$, \label{cond no quotient F2}
\item $1_R$ is the sum of two invertible elements, \label{1 sum of inv}
\item Every element of $R$ is the sum of two invertible elements. \label{every element sum of two units}
\end{enumerate}
\end{proposition}
\begin{proof}
The implication \textit{(\ref{every element sum of two units})} $\implies$ \textit{(\ref{1 sum of inv})} is trivial. To see that \textit{(\ref{1 sum of inv})} implies \textit{(\ref{cond no quotient F2})}, suppose $1_R = u + v$ with $u,v \in R^{\times}$. Suppose $R$ has a quotient $\phi \colon R \onto \F_2$. Then $\phi(u) = \phi(v) = 1$, since $u$ and $v$ are invertible and hence so is their image. But then $\phi(1_R) = \phi(u)+\phi(v) = 0$, a contradiction. Finally, to see that \textit{(\ref{cond no quotient F2})} implies \textit{(\ref{every element sum of two units})}, let $J = J(R)$, the Jacobson radical of $R$. It is characterised as the ideal \[J(R) = \left\{ x \in R \mid 1 + yx \in R^{\times}, \text{ for all } y \in R \right\}.\]
First we claim that for $r\in R$, if $r + J$ is a sum of units in $R/J$, then $r$ is a sum of units in $R$. It then suffices to show that every element of $R/J$ is a sum of units.

First of all, if $a + J$ is invertible in $R/J$, then there is some $x \in R$ such that $xa + J = 1 + J$, and in particular, $xa-1 \in J$, i.e.\ $xa \in R^{\times}$ by definition of $J$. It then follows that $a \in R^{\times}$. 

Suppose now that $r + J$ is the sum of two units in $R/J$, i.e.\ $r + J = u + v + J$ for some $u+J, v+J \in (R/J)^{\times}$. Then there is some $j \in J$ such that $r = u + v + j$, and since $u$ and $v$ are invertible by the previous paragraph, we have that \[r = u + v + j = u + v(1+v^{-1}j).\]
Note that $1+v^{-1}j \in R^{\times}$ since $j \in J$, and hence also $v(1+v^{-1}j) \in R^{\times}$. In particular, $r$ is the sum of two units, as desired.

We now show that every element in $R/J$ is a sum of two units. Since $R/J$ is semisimple, by the Wedderburn--Artin theorem $R/J$ decomposes as 
\[R/J \cong \prod_{i = 1}^{n} \mathrm{M}_{n_i}(D_i), \]
with each $D_i$ a division ring. Since $R$ is assumed to have no quotient $\F_2$, in particular no $\mathrm{M}_{n_i}(D_i)$ in the above decomposition is isomorphic to $\F_2$. By \cite{Zelinsky}, not every element of a matrix ring $\mathrm{M}_n(D)$ for $D$ a division ring is a sum of two units if and only if $n = 1$ and $D = \F_2$. Hence for every $i$, every element $r \in \mathrm{M}_{n_i}(D_i)$ is the sum of two units, and by \cite[Proposition 1.2]{Goldsmith}, this property is preserved under taking direct products of rings. In particular, every element of $R/J$ is a sum of two units, and the conclusion follows.
\end{proof}

\begin{remark}
\label{rk F2}
The equivalence of \textit{(\ref{cond no quotient F2})} and \textit{(\ref{every element sum of two units})} is studied more extensively in the literature.  For example, the authors of \cite{Khurana} also show the result for right self-injective rings, using Kaplansky-style decompositions of such rings. The authors of \emph{loc.\ cit.\ }also mention the existence of (non-commutative) von Neumann regular rings in which $2$ is invertible (and hence $1$ is a sum of two units, and the ring has no quotient $\F_2$) in which not every element is the sum of two units, attributed to Bergman (see \cite{Handelman}). 

For our purposes we are mostly interested in the equivalence of \textit{(\ref{cond no quotient F2})} and \textit{(\ref{1 sum of inv})}, since \textit{(\ref{cond no quotient F2})} is a necessary condition for the conclusion of \Cref{spherical main} to hold, by \Cref{remark Tits} and functoriality of $U^+_R$ in $R$, while \textit{(\ref{1 sum of inv})} is a sufficient condition. However, we were unable to find a result in the literature which completely classifies the rings $R$ for which this equivalence holds, even when restricted to the case of commutative rings. 
\end{remark}

\section{\texorpdfstring{Presentation of $\ov{U}^+_R$}{Presentation of U+(R)}}
\label{section presentation U}
In this section, we now prove \Cref{corollary C} and \Cref{corollary injective}. We first state the following result, due to Stein (\cite{Stein}). See also \cite[Theorem 5.5]{MMinject}. 
\begin{lemma}[{\cite{Stein}}]
\label{U+ spherical is U+bar}
Let $R$ be a local domain and let $A$ be spherical. Then $\varphi_R \colon G_R \to G^{\min}_R$ is a bijection. In particular, $U^{\pm}_R \cong \ov{U}^{\pm}_R$.
\end{lemma}

The following theorem shows the reverse implication in \Cref{corollary C}.
\begin{theorem}
\label{3 spherical implies}
Suppose $A$ is a $3$-spherical simply laced GCM, and let $R$ be a local Bézout domain with residue field different from $\F_2$. Then $\ov{U}^+_R$ is the amalgamated product of its subgroups $\ov{U}_{JR}^+$ with $J\subseteq I$ and $\abs{J} \leq 2$, along the inclusions $\ov{U}_{J_1 R}^+ \leqslant \ov{U}^+_{J_2 R}$ for all $J_1 \subseteq J_2$.
\end{theorem}
\begin{proof}
By \Cref{main thm in text}, it suffices to show the statement for all rank $3$ GCM's $A$ of simply laced spherical type. In this case, $\ov{U}^+_R \cong U^+_R$ by \Cref{U+ spherical is U+bar}. We distinguish three cases for $A$. 

If $A$ is indecomposable, then it is of type $A_3$, and the statement follows from \Cref{case A3}, since by \Cref{exists element a+1 a2+1}, $1_R$ is a sum of two units.  

If $A$ is decomposable, then $U^+_R \cong U^+_{A' R} \times (R,+)$, where $A'$ is either of type $A_2$ or of type $A_1 \times A_1$, i.e.\ \[U^+_R \cong \left\{ \begin{psmallmatrix} 1 & \ast & \ast \\ 0 & 1 & \ast \\ 0 & 0 & 1\end{psmallmatrix} \in \SL_3(R)\right\} \times (R,+) \quad \text{ or } \quad U^+_R \cong (R,+)\times(R,+)\times (R,+).\] 
It is then immediate from the relations \eqref{commutation KM min} that $U^+_R$ has the desired presentation.
\end{proof}

We now prove \Cref{corollary injective}. In \cite[Theorem A]{MMinject}, this result was already proven in the more general setting of arbitrary $2$-spherical GCM's, but as a consequence of our work we obtain an alternative proof in the $3$-spherical simply laced case. 
\begin{proposition}
\label{injectivity}
Let $A = (A_{ij})_{i,j \in I}$ be $3$-spherical and simply laced. Let $R$ be a local Bézout domain with residue field different from $\F_2$, and with field of fractions $\K$. Then the map
\[\varphi_R \colon G_R \to G^{\min}_R \subseteq G_{\K}\]
is injective.
\end{proposition}
\begin{proof}
By \Cref{3 spherical implies}, $\ov{U}^+_R$ is the free product of the root subgroups $U_{\alpha R}$, for the roots $\alpha \in \Psi$, where
\[ \Psi := \left\{\alpha \in \Phi \mid \exists i,j \in I \text{ such that }\alpha \in (\N \alpha_i + \N\alpha_j ) \cap \Phi^+\right\}, \]
subject to the relations \eqref{commutation KM min} for all pairs $\{\alpha, \beta \}\subseteq \Psi$ with $\alpha \neq \pm \beta$. But these relations are all already satisfied in $U^+_R$, and since moreover $U^+_R$ is generated by all the simple root subgroups $U_{\alpha_i R}$ by \Cref{generated by simples 2 spherical}, it follows that $\varphi_R \hspace{-3pt}\mid_{U^+_R}$ is injective. 

The proof of \cite[Theorem 5.5]{MMinject}, now applies \emph{verbatim}.
\end{proof}
We now show the forward implication in \Cref{corollary C}.
\begin{proposition}
\label{3 spherical necessary}
Suppose $A = (A_{ij})_{i,j \in I}$ is simply laced and $R$ is a local Bézout domain with residue field different from $\F_2$. If $\ov{U}^+_R$ is isomorphic to the amalgamated product of its subgroups $\ov{U}_{JR}^+$ with $J \subseteq I$ and $\abs{J}\leq 2$ along the inclusions $\ov{U}^+_{J_1 R } \leqslant \ov{U}^+_{J_2R}$ for $J_1 \subseteq J_2$, then $A$ is $3$-spherical.
\end{proposition}
\begin{proof}
By \Cref{injectivity}, $\varphi_R$ is injective, and hence $\ov{U}^+_R$ is isomorphic to $U^+_R \leqslant G_R$. The result then immediately follows from \cite[Theorem E]{CapMar26}.
\end{proof}

\Cref{corollary C} now follows from \Cref{3 spherical implies} and \Cref{3 spherical necessary}.

Finally, note that the stabilisers in $\ov{B}^+_R$ of the $J$-residues $\ov{U}^+_R c_0^- = \ov{B}^+_R c_0^- \subseteq \mc{C}^-(G^{\min}_R)$ are given by 
\[ \ov{B}^+_R \cap \ov{P}^-_{JR} = ( \ov{U}^+_R \cap \ov{P}^-_{JR})T_R = \ov{U}^+_{JR}T_R, \]
where the last equality follows by \eqref{intersection B+ P-}. We then deduce the following.

\begin{corollary}
Let $R$ be a local Bézout domain with residue field different from $\F_2$. Let $A$ be a simply laced GCM. Then $\ov{B}^+_R$ is the amalgamated product of its subgroups $\ov{U}_{JR}^+ T_R$ for all $J \subseteq I$ with $\abs{J} \leq 2$, along the inclusions $\ov{U}_{J_1 R}^+ T_R \leqslant \ov{U}_{J_2 R}^+ T_R$ for $J_1 \subseteq J_2$, if and only if $A$ is $3$-spherical.
\end{corollary}
Note also that applying the Cartan--Chevalley involution $\omega_R$, the analogue of any result stated in this paper for $\ov{U}^+_R$ and $U^+_R$ also holds for $\ov{U}^-_R$ and $U^-_R$.

\bibliographystyle{alpha}
\bibliography{references}

\end{document}